\documentclass[11pt,twoside, a4paper]{amsart}
\usepackage[utf8]{inputenc}
\usepackage[T1]{fontenc}
\usepackage{geometry}
\usepackage{amsmath}
\usepackage{amsfonts}
\usepackage{amssymb}
\usepackage{tikz}
\usepackage{anyfontsize}
\usepackage{graphicx}
\usepackage{rotating}
\usepackage{amsthm}
\usepackage{yfonts}
\usepackage{mathtools}
\usepackage{pdfpages}
\usepackage{appendix}
\usepackage{float}
\usepackage{biblatex}
\usepackage{hyperref}
\usepackage{cleveref}
\hypersetup{colorlinks=true, urlcolor=blue}
\usepackage{lastpage}
\usepackage{fancyhdr}
\usepackage{listings}
\usepackage{xcolor}

\definecolor{codegreen}{rgb}{0,0.6,0}
\definecolor{codegray}{rgb}{0.5,0.5,0.5}
\definecolor{codepurple}{rgb}{0.58,0,0.82}
\definecolor{backcolour}{rgb}{0.95,0.95,0.92}

\lstdefinestyle{mystyle}{
    backgroundcolor=\color{backcolour},   
    commentstyle=\color{codegreen},
    keywordstyle=\color{magenta},
    numberstyle=\tiny\color{codegray},
    stringstyle=\color{codepurple},
    basicstyle=\ttfamily\footnotesize,
    breakatwhitespace=false,         
    breaklines=true,                 
    captionpos=b,                    
    keepspaces=true,                 
    numbers=left,                    
    numbersep=5pt,                  
    showspaces=false,                
    showstringspaces=false,
    showtabs=false,                  
    tabsize=2
}

\newcommand{\B}{\mathbb{B}}
\newcommand{\N}{\mathbb{N}}
 
\newcommand{\R}{\mathbb{R}}

\renewcommand{\S}{\mathbb{S}}
\renewcommand{\epsilon}{\varepsilon}

\newtheorem{thm}{Theorem}
\newtheorem{prop}[thm]{Proposition}
\newtheorem{corollaire}[thm]{Corollary}
\newtheorem{lemma}[thm]{Lemma}
\theoremstyle{definition}

\newtheorem{defn}[thm]{Definition}
\newtheorem{rem}[thm]{Remark}
\newtheorem{expl}[thm]{Example}

\newtheorem{conj}[thm]{Conjecture}

\title {Critical lengths of Steklov eigenvalues of hypersurfaces of revolution in Euclidean space 
}

\author[M\'etras]{Antoine M\'etras}
\address{University of Bristol,
School of Mathematics,
Fry Building,
Woodland Road,
Bristol, 
BS8 1UG, U.K.}
\email{antoine.metras@bristol.ac.uk}

\author[Tschanz]{Léonard Tschanz}
\address{Institut de Mathémtiques, Université de Neuchâtel, rue Emile-Argand 11, Neuchâtel, Switzerland.}
\email{leonard.tschanz@unine.ch}

\date{}

\begin{document}

\begin{abstract}
    We study the Steklov problem on hypersurfaces of revolution with two  boundary components in Euclidean space. In \cite{T_revolution}, the phenomenon of critical length, at which a Steklov eigenvalue is maximized, was exhibited and multiple questions were raised. 
    In this article, we conjecture that, in any dimension, there is a finite number of infinite critical length. 
    To investigate this, we develop an algorithm to efficiently perform numerical experiments, providing support to our conjecture. 
    Furthermore, we prove the conjecture in dimension $n = 3$ and $n = 4$.
    \medskip
    
    \medskip
    
\end{abstract}

\maketitle

\section{Introduction}

Let $(M, g)$ a smooth compact connected Riemannian manifold  of dimension $n \ge 2$ with smooth boundary $\Sigma$, the Steklov problem on $(M, \Sigma)$ consists in finding the real numbers $\sigma$ and the functions $f : M \longrightarrow \R$ such that 
\begin{align*}
    \left\{
    \begin{array}{c}
    \Delta f = 0 \mbox{ in } M \\
    \partial_\nu f=\sigma f \mbox{ on } \Sigma,
    \end{array}
    \right.
\end{align*}
 where $\nu$ denotes the outward normal on $\Sigma$. Such a $\sigma$ is called a Steklov eigenvalue of $(M, g)$. It is well known that the Steklov spectrum forms a discrete sequence $0= \sigma_0(M,g) < \sigma_1(M,g) \le \sigma_2(M,g) \le \ldots \nearrow \infty$, where each eigenvalue is repeated with its multiplicity, which is finite.
\medskip

Since we are interested in upper bounds for the Steklov eigenvalues of  a family of manifolds, we have to add some geometric constraints on the manifolds. Indeed, \parencite[Theorem 1.1]{CEG2} states that if $(M, g)$ is a compact connected manifold of dimension $n \ge 3$ with boundary $\Sigma$, then there exists a family of Riemannian metrics $(g_\epsilon)$ conformal to $g$ and that agrees with $g$ on $\Sigma$ such that 
\begin{align*}
    \sigma_1(M, g) \longrightarrow \infty \; \mbox{ as } \; \epsilon \longrightarrow 0.
\end{align*}
A class of manifolds that was recently investigated consists of warped products of the form $[0, L] \times \S^{n-1}$, where $L$ is a positive real number and $\S^{n-1}$ is the unit sphere of dimension $n-1$, see for instance \cite{ FTY, FS,   XC, XC2}. Some authors studied the particular case where these warped products manifolds can be seen as hypersurfaces of revolution in Euclidean space, see \cite{CGG, CV, T_revolution}. In this article, we will focus in these hypersurfaces of revolution in Euclidean space.  Let us start by recalling what they are.

\begin{defn} \label{defn : revolution conj}
A $n$-dimensional compact hypersurface of revolution $(M,g)$ in Euclidean space  with two boundary components 
is the warped product $M=[0, L] \times \S^{n-1}$ endowed with the Riemannian metric 
\begin{align*}
    g(r,p) = dr^2 + h^2(r)g_0(p),
\end{align*}
where 
\begin{enumerate}
    \item The metric $g_0$ is the canonical one on  $\S^{n-1}$,
    \item The function $h : [0, L] \longrightarrow \R_+^*$ is  smooth and  satisfies $|h'(r)| \le 1$ for all $r \in [0, L]$.
\end{enumerate}
This assumption on $h$ is a consequence of the fact that the hypersurface lives in Euclidean space, see \parencite[Section $3$]{CGG} for more details.
\medskip

As it was done in \cite{CGG, CV, T_revolution}, in this article we will only consider hypersurfaces of revolution whose boundary components are isometric to a unit $(n-1)$-sphere. Therefore, we will always assume that   $h(0)=h(L)=1$.
\medskip

If $M = [0, L] \times \S^{n-1}$ and $h : [0, L] \longrightarrow \R_+^*$ satisfy the properties above, we say that $M$ is a \textit{hypersurface of revolution}, that  $g(r, p) = dr^2 + h^2(r)g_0(p)$ is a \textit{metric of revolution on $M$ induced by $h$} and we call the number $L$ the \textit{meridian length} of the manifold $M$.
\end{defn}

We refer to \cite{T_revolution} for an overview of what is already known concerning upper and lower bounds for the Steklov spectrum of hypersurfaces of revolution. Here, because we want to investigate \parencite[Question 22]{T_revolution}, we will focus on the upper bounds. As such, the meridian length $L$ plays an important role. Indeed, it was shown in \cite{T_revolution} that when maximizing the $k$-th Steklov eigenvalue, for each given meridian length $L$, a sharp upper bound exists depending only on $L$. The problem is then reduced into studying this upper bound in function of $L$. One can show that for some value of $k$, the global maximum is achieved for some finite number $L_k \in \mathbb{R}_+^*$, while for other $k$, the supremum is achieved at infinity. This motivates the following definitions.

\begin{defn} \label{def : critical length conj}
Let $n \ge 3$ and $k \ge 1$ be integers and $L \in \R_+^*$.
\begin{enumerate}
    \item  We write $B_n^k(L) < \infty$ for the sharp upper bound for the $k$th eigenvalue of a hypersurface of revolution of dimension $n$ and meridian length $L$.
    \item We write $B_n^k := \sup_{L \in \R_+^*} \{B_n^k(L)\}$.
    \item We say that $L_k \in \R_+^*$ is a finite critical length associated with $k$, and that $k$ has a finite critical length, if we have $B_n^k = B_n^k(L_k)$. 
    \item We say that $k$ has a critical length at infinity, or an infinite critical length, if it satisfies $B_n^k= \lim_{L\to \infty} B_n^k(L)$.
\end{enumerate}
\end{defn}

When we look at the literature, we can see that the phenomenon of finite/infinite critical lengths is intriguing. Indeed, 
\begin{enumerate}
    \item In dimension $n=2$:
    \begin{enumerate}
        \item All surfaces of revolution of $\mathbb{E}^3$ with one boundary component isometric to the unit circle are isospectral to the unit disk \parencite[Proposition 1.10]{CGG}. Therefore, for any fixed $k$, all meridian lengths $L$ are critical for $k$ in the sense that all of them achieve the maximum;
        \item In the case of surfaces of revolution with two boundary components, \parencite[Theorem 1.1]{FTY} states that $k = 2$ has an infinite critical length while every other eigenvalues have a finite critical length.
    \end{enumerate}
    \item In dimension $n \ge 3$:
    \begin{enumerate}
        \item If $(M, g)$ is a hypersurface of revolution of the Euclidean space with one boundary component isometric to the unit sphere, then for all $k \ge 1$, the $k$th eigenvalue has an infinite critical length, see \parencite[Proof of Theorem 1]{CV};
        \item If $(M, g)$ is a hypersurface of revolution of the Euclidean space with two boundary components each isometric to the unit sphere, the set of eigenvalues which have a finite critical length is non empty. Indeed, \parencite[Corollary 4]{T_revolution} guarantees that $k=1$ has a finite critical length.  Moreover, the set of eigenvalues which have a critical length at infinity is also non empty. Indeed, in  \parencite[Section 6.1]{T_revolution} it was shown that $k=2$ has an infinite critical length.
    \end{enumerate}
\end{enumerate}
This consideration naturally calls for further research on finite/infinite critical lengths in the case $n \geq 3$ and with two boundary components.
\medskip

We state here \parencite[Theorem 9]{T_revolution}:
\begin{thm} \label{thm : principal principal 4}
Let $n \ge 3$. Then there exist infinitely many $k \in \N$ which have a finite critical length associated with them. 
Moreover, if we call  $(k_i)_{i=1}^\infty \subset \N$ the increasing sequence of such $k$ and if we call $(L_i)_{i=1}^\infty$ the associated sequence of finite critical lengths, then we have 
\begin{align*}
    \lim_{i\to \infty }L_i=0.
\end{align*}
\end{thm}

This result immediately raises the following open question (\parencite[Question 22]{T_revolution}):

\begin{center}
    \textit{Given $n \ge 3$, are there finitely or infinitely many $k \in \N$ such that $k$ has a critical length at infinity?} 
\end{center}

To investigate this question, we developed the following tool:

\begin{thm} \label{thm : thm principal 3}
Let $(M=[0, L] \times \S^{n-1}, g)$ be a hypersurface of revolution  in Euclidean space with two boundary components each isometric to  $\S^{n-1}$, with $n \ge 3$, and let $k \ge 1$. Then there exists an algorithm, called  \emph{extension process}, {that computes the sharp upper bound $B_n^k(L)$, depending only on $n, k$ and $L$, and shows that the upper bound is never achieved :
\begin{align*}
    \sigma_{k}(M,g) < B_n^k(L).
\end{align*}
Moreover, for any $\epsilon >0$, we construct a metric of revolution $g_\epsilon$ on $M$ such that $\sigma_k(M,g_\epsilon) > B_n^k(L)-\epsilon$}.

\end{thm}

Using this result, we obtain two corollaries:

\begin{corollaire} \label{cor : du principal 3}
Let $n \ge 3$ and $k \ge 1$. Then there exists a bound $B_n^k < \infty$ such that for all hypersurfaces of revolution $(M, g)$  in Euclidean space with two boundary components each isometric to  $\S^{n-1}$, we have
\begin{align*}
    \sigma_k(M,g) < B_n^k, 
\end{align*}
given by 
\begin{align*}
    B_n^k := \sup \{B_n^k(L) : L \in \R_+^* \}.
\end{align*}
Moreover, this bound is sharp: for each $\epsilon >0$, there exists a  hypersurface of revolution $(M_\epsilon, g_\epsilon)$ such that $\sigma_k(M_\epsilon, g_\epsilon) > B_n^k- \epsilon$.
\end{corollaire}

\begin{corollaire} \label{cor : deuxieme cor du principal 3}
Let $(M_i, g_i)_{i=1}^\infty$ be a family of hypersurfaces of revolution  in Euclidean space with two boundary components each isometric to  $\S^{n-1}$, where $M_i=[0, L_i]\times \S^{n-1}$, $n\ge 3$. Let $k \in \N$. Let us suppose that $L_i \underset{i \to \infty}{\longrightarrow} 0$. Then we have 
\begin{align*}
    \sigma_k(M_i, g_i) \underset{i \to \infty}{\longrightarrow} 0.
\end{align*}
\end{corollaire}

\begin{rem}
Corollary \ref{cor : deuxieme cor du principal 3} is already contained in Proposition 3.3 of \cite{CGG}.
\end{rem}

In this paper we do not answer the open question completely (we do so in the case of dimension $n=3$ and $n=4$) but we perform some numerical experiments to clarify the phenomenon of finite / infinite critical lengths. Here is an overview of the experiments results:
\begin{align*}
    \begin{array}{cl}
     n  & \mbox{Result of the program investigations}  \\
     \hline
     3  & \mbox{From $k=18$ to $k=45'601$, only finite critical lengths found.}\\
     4 & \mbox{From $k=408$ to $k=47'641$, only finite critical lengths found.} \\
     5 & \mbox{From $k=8'400$ to $k=195'423$, only finite critical lengths found.} \\
     6 & \mbox{From $k=21'112$ to $k=610'973$, only finite critical lengths found.}
\end{array}
\end{align*}

Therefore, we propose the following conjecture:
\begin{conj} \label{conj : critical length}
    Let $n \ge 2$ be an integer. Then there exists a constant $K=K(n) \in \N$ such that for every $k \ge K$, the $k$th eigenvalue has an associated finite critical length.
\end{conj}
While investigating other questions and topics, Fan et al. proved this conjecture in the special case of dimension $2$, see \parencite[Theorem 1.1]{FTY}.  This conjecture is still to be proved or refuted in general, but in this article, we are able to prove it in the case of dimensions $n=3$ and $n=4$.
\begin{thm} \label{thm : low dimension}
    Let $n \in \{3, 4\}$. Then there exists a constant $K(n) \in \N$ such that for every $k \ge K$, the $k$th eigenvalue has a finite critical length.
\end{thm}

\textbf{Plan of the paper.} In \Cref{sect : context}, we recall the context and fix some notation we will use throughout the paper. In \Cref{subsect : ext process}, we describe the extension process and prove \Cref{cor : du principal 3} and \Cref{cor : deuxieme cor du principal 3}, so that we can formulate the question and the conjecture properly in \Cref{sect : conj}. In \Cref{sect : low dimension}, we prove \Cref{thm : low dimension}, namely we solve the question for the case of dimension $n=3$ and $n=4$. Finally, in Appendix \ref{sect : plot simple} we plot the sharp upper bound as a function of $L$, then we add on these plots the mixed Steklov-Dirichlet and Steklov-Neumann eigenvalues in Appendix  \ref{sect : plot mixte} and we give our codes  in Appendix \ref{appendix : python codes}.
\medskip


\section{Hypersurfaces of revolution and mixed problems} \label{sect : context}

As explained in \cite{T_revolution}, maximizing the Steklov eigenvalues of hypersurfaces of revolution is deeply linked to the comprehension of the mixed Steklov-Dirichlet and Steklov-Neumann problem on annular domains. We recall what these mixed problems are in this section, as well as recalling what these links are.

\subsection{Characterization of the Steklov eigenvalues and eigenfunctions}

For $(M, g)$ a Riemannian manifold with smooth boundary $\Sigma$, we can characterize its $k$th Steklov eigenvalue by:
\begin{align} \label{form : car var conj}
    \sigma_k(M, g) = \min \left\lbrace R_g(f)  : f \in H^1(M), \; f \perp_\Sigma f_0, f_1, \ldots, f_{k-1} \right\rbrace,
\end{align}
where 
\begin{align*}
    R_g(f) = \frac{\int_M |\nabla f|^2 dV_g}{\int_\Sigma |f|^2 dV_\Sigma} 
\end{align*}
 is the Rayleigh quotient  and
\begin{align*}
    f \perp_\Sigma f_i \iff \int_\Sigma f f_i dV_\Sigma =0,
\end{align*}
with the $f_i$ being $\sigma_i$-eigenfunctions, $i = 1, \dots, k-1$.

Similarly, we have the usual Min-Max characterization 
\begin{align} \label{form : min max}
    \sigma_k(M, g) = \min_{S \in S_{k+1}} \max_{f \in S \backslash \{0\} } \left\{ R_g(f) \right\},
\end{align}
where $S_{k+1}$ is the family of all subspaces of dimension $k+1$ of $H^1(M)$.
\medskip

In the special setting of this article, where $(M, g)$ is a hypersurface of revolution, the corresponding eigenfunctions have a special expression. 
We denote by $0 =\lambda_0 < \lambda_1 \le \lambda_2 \le \ldots \nearrow \infty$ the spectrum of the Laplacian on $(\S^{n-1}, g_0)$ and we consider $(S_j)_{j=0}^\infty$ an orthonormal basis of eigenfunctions associated with it. 
\begin{prop} \label{prop : sep variables conj}
Let $(M, g)$ be a hypersurface of revolution  as above. Then each eigenfunction on $M$ can  be written as $f_l(r, p) = u_l(r)S_j(p)$, where  $u_j$ is a smooth function on $[0, L]$.
\end{prop}

This expression for the eigenfunctions is true in the more general case of warped product manifolds (and therefore for hypersurfaces of revolution) and it is often used, see for example  \parencite[Remark 1.3]{DHN},  \parencite[Lemma 3]{Esc}, \parencite[Proposition $3.16$]{T2} or  \parencite[Proposition $9$]{XC}.

\subsection{Mixed problems on annular domains}

Let $\B_1$ and $\B_R$ be the balls in $\R^n$, with $R >1$ and $n \ge 3$, centered at the origin. The annulus $A_R$ is defined as follows: $A_R = \B_R \backslash \overline{\B}_1$. We say that this annulus is of inner radius $1$ and outer radius $R$. This particular kind of domains  shall be useful in this article.
\medskip

For such domains, it is possible to compute explicitly $\sigma_{(k)}^D(A_R)$, which is the $(k)$th eigenvalue of the Steklov-Dirichlet problem on $A_R$, counted without multiplicity.
\medskip

We state here Proposition $4$ of \cite{CV}:
\begin{prop} \label{prop : SD conj}
For $A_R$ as above, consider the Steklov-Dirichlet problem
\begin{align*}
    \left\{
    \begin{array}{ll}
       \Delta f = 0  & \mbox{ in } A_R  \\
        \partial_\nu f = \sigma f & \mbox{ on } \partial \B_1 \\
        f = 0 & \mbox{ on } \partial \B_R.
    \end{array}
    \right.
\end{align*}
Then, for $k \ge 0$, the $(k)$th eigenvalue (counted without multiplicity) of this problem is 
\begin{align*}
    \sigma_{(k)}^D(A_R) = \frac{(k+n-2)R^{2k+n-2}+k}{R^{2k+n-2}-1}.
\end{align*}
\end{prop}


It is also possible to get the expression of the eigenfunctions of the Steklov-Dirichlet problem on an annular domain.

\begin{lemma} \label{lem : Dirichlet conj}
Each eigenfunction $\varphi_l$ of the Steklov-Dirichlet problem on the annulus $A_{R}$ can be expressed as $\varphi_l(r,p)=\alpha_l(r)S_l(p)$, where $S_l$ is an eigenfunction for the $l^{th}$ harmonic of the sphere $\mathbb{S}^{n-1}$.
\end{lemma}

Similarly we can compute explicitly $\sigma_{(k)}^N(A_R)$, which is the $(k)$th eigenvalue of the Steklov-Neumann problem on $A_R$, counted without multiplicity.
\medskip

We state now Proposition $5$ of \cite{CV}:
\begin{prop} \label{prop : SN conj}
For $A_R$ as above, consider the Steklov-Neumann problem
\begin{align*}
    \left\{
    \begin{array}{ll}
       \Delta f = 0  & \mbox{ in } A_R  \\
        \partial_\nu f = \sigma f & \mbox{ on } \partial \B_1 \\
        \partial_\nu f = 0 & \mbox{ on } \partial \B_R.
    \end{array}
    \right.
\end{align*}
Then, for $k \ge 0$, the $(k)$th eigenvalue (counted without multiplicity) of this problem is 
\begin{align*}
    \sigma_{(k)}^N(A_R) = k \frac{(k+n-2)(R^{2k+n-2}-1)}{kR^{2k+n-2}+k+n-2}.
\end{align*}
\end{prop}

In the same manner as before, we have the following:

\begin{lemma} \label{lem : Neumann conj}
Each eigenfunction $\phi_l$ of the Steklov-Neumann problem on the annulus $A_{R}$ can be expressed as $\phi_l(r,p)=\beta_l(r)S_l(p)$, where $S_l$ is an eigenfunction for the $l^{th}$ harmonic of the sphere $\mathbb{S}^{n-1}$.
\end{lemma}

\subsection{Multiplicity of the eigenvalues}

Since we will have to deal with several problems and the multiplicity of the eigenvalues, we start by giving some notation, summarized in the table below:
\begin{align*}
    \begin{array}{l|cc}
      \mbox{Problem} & k \mbox{th eigenvalue} & k\mbox{th eigenfunction} \\
      \hline
      \mbox{Laplace} & \lambda_k & S_k  \\
      \mbox{Steklov} & \sigma_k & f_k=u_kS_j  \\
      \mbox{Steklov-Dirichlet} & \sigma_k^D & \varphi_k = \alpha_kS_k \\
      \mbox{Steklov-Neumann} & \sigma_k^N & \phi_k = \beta_kS_k
    \end{array}
\end{align*}
\medskip
 We also write $\lambda_{(k)}, \sigma_{(k)}, \sigma_{(k)}^D, \sigma_{(k)}^N$ for the $(k)$th eigenvalue counted without multiplicity.

\medskip

In the case of the classical Laplacian problem $\Delta S = \lambda S$ on  $(\S^{n-1}, g_0)$, it is known that the set of eigenvalues is $\{ \lambda_{(k)} = k(n+k-2) : k \ge 0 \}$, see \parencite[Page 160-162]{BGM}. Besides, the multiplicity $m_0$ of $\lambda_{(0)}=0$ is $1$ and the multiplicity of $\lambda_{(k)}$ is 
\begin{align} \label{frm : multiplicite conj}
    m_k:=\frac{(n+k-3)(n+k-4) \ldots n(n-1)}{k!}(n+2k-2).
\end{align}
 
Thus, given $k \ge 0$, there are $m_k$ independent functions $S_k^1, \ldots, S_k^{m_k}$ which satisfy
\begin{align*}
    \Delta S_k^i = \lambda_{(k)} S_k^i,  \mbox{ for all } i=1, \ldots, m_k.
\end{align*}
Moreover, for both the Steklov-Dirichlet and Steklov-Neumann problems on annular domains, the multiplicity of the $(k)$th eigenvalue is exactly $m_k$, as stated by \parencite[Proposition 3]{CV}. 

Hence, using the notation above, given $k \ge 0$, 
\begin{enumerate}
    \item There are  exactly $m_k$ linearly independent Steklov-Dirichlet eigenfunctions associated with  $\sigma_{(k)}^D(A_{1+L/2})$, that can be written $\varphi_k^i(r,p) = \alpha_{k}(r)S_{k}^i(p), \; i = 1, \ldots, m_k$.
    \item There are  exactly $m_k$ linearly independent Steklov-Neumann eigenfunctions associated with  $\sigma_{(k)}^N(A_{1+L/2})$, that  can be written $\phi_k^i(r,p) = \beta_{k}(r)S_k^i(p), \; i = 1, \ldots, m_k$.
\end{enumerate}

\section{Extension process} \label{subsect : ext process}

It is natural to wonder if we can get a general formula giving a sharp upper bound for the $k$th Steklov eigenvalue of $(M, g)$, depending on $L$ and $n$ or even only on $n$.
The expression for such a general bound is difficult to obtain. However, one can introduce a process, that we call \textit{extension process}, leading to a sharp bound $B_n^k(L)$ such that $\sigma_k(M,g) < B_n^k(L)$. This process is the following.
\medskip

Let $(M=[0, L] \times \S^{n-1}, g)$ be a hypersurface of revolution and let $k \in \N^*$. We define 
\begin{align*}
   l_0 := \max \left\{ l \in \N : \sum_{i=0}^l m_i \le k \right\}.
\end{align*}

Let us now consider the finite set 
\begin{align*}
    E:=\left\{ \sigma_{(0)}^D(A_{1+L/2}), \ldots, \sigma_{(l_0)}^D(A_{1+L/2}), \sigma_{(1)}^N(A_{1+L/2}), \ldots, \sigma_{(l_0+1)}^N(A_{1+L/2}) \right\}.
\end{align*}

One can rearrange this set in ascending order, i.e we choose $\pi \in Sym(E)$ such that the finite sequence 
\begin{align*}
      \left(  \pi(\sigma_{(0)}^D(A_{1+L/2}) ),   \ldots,  \pi( \sigma_{(l_0)}^D(A_{1+L/2})),   \pi ( \sigma_{(1)}^N(A_{1+L/2}) ),   \ldots,   \pi(\sigma_{(l_0+1)}^N(A_{1+L/2}) ) \right)
\end{align*}
is increasing. For practical reasons, we rename this sequence 
\begin{align*}
    \left( \nu_0,  \ldots,  \nu_{l_0},  \nu_{l_0+1},  \ldots,  \nu_{2l_0+1} \right).
\end{align*}

Let us now consider the corresponding multiplicities. For $i \in \{0, \ldots, 2l_0+1\}$, we write $ \mu_i $ for the multiplicity of $\nu_i$. This gives us a finite sequence 
\begin{align*}
     ( \mu_0, \ldots, \mu_{l_0}, \mu_{l_0+1}, \ldots, \mu_{2l_0+1}).
\end{align*}

We define 
\begin{align*}
    l_1 := \min \left\{ l \in \{0, \ldots, 2l_0+1\} : \sum_{i=0}^l \mu_i \ge k \right\},
\end{align*}
and we will prove that 
\begin{align*}
    B_n^k(L):= \nu_{l_1}
\end{align*}
is a sharp upper bound for $\sigma_k(M,g)$.

\begin{expl}
We take the case $n=5$ and $L=1$, and we choose $k=127$. Using the formula 
\begin{align*}
    m_i=\frac{(n+i-3)(n+i-4) \ldots n(n-1)}{i!}(n+2i-2),
\end{align*}
 we have $m_0=1, \; m_1=5, \; m_2=14, \; m_3=30, \; m_4=55, \; m_5=91$. 
 Therefore, $\sum_{l=0}^4 m_l= 105 \le k=127$, but $\sum_{l=0}^5 m_l = 196  > k$, which means that we have $l_0=4$.

We consider the set 
\begin{align*}
    E= \{ & \sigma_{(0)}^D(A_{3/2}), \sigma_{(1)}^D(A_{3/2}),  \sigma_{(2)}^D(A_{3/2}), \sigma_{(3)}^D(A_{3/2}), \sigma_{(4)}^D(A_{3/2}),  \\
    & \sigma_{(1)}^N(A_{3/2}), \sigma_{(2)}^N(A_{3/2}), \sigma_{(3)}^N(A_{3/2}),      \sigma_{(4)}^N(A_{3/2}), \sigma_{(5)}^N(A_{3/2}) \}.
\end{align*}

Using Propositions \ref{prop : SD conj} and \ref{prop : SN conj}, we can determine the value of the $10$ numbers belonging to the set $E$, and arrange them in an ascending order. We get 
\begin{align*}
    (& \sigma_{(1)}^N(A_{3/2}) =:  \nu_0 \approx 2.27, \sigma_{(2)}^N(A_{3/2}) =: \nu_1 \approx 4.11, \sigma_{(0)}^D(A_{3/2}) =: \nu_2 \approx 4.26, \\
    & \sigma_{(1)}^D(A_{3/2}) =: \nu_3 \approx 4.76, \sigma_{(2)}^D(A_{3/2}) =: \nu_4 \approx 5.44,  \sigma_{(3)}^N(A_{3/2})=: \nu_5 \approx 5.55, \\
    & \sigma_{(3)}^D(A_{3/2}) =: \nu_6 \approx 6.24, \sigma_{(4)}^N(A_{3/2})=: \nu_7 \approx 6.78, \sigma_{(4)}^D(A_{3/2}) =: \nu_8 \approx 7.13, \\
    & \sigma_{(5)}^N(A_{3/2}) =:\nu_9 \approx 7.89 ).
\end{align*}
Writing $\mu_i$ for the multiplicity of $\nu_i$, we can associate to $(\nu_i)_{i=0}^9$ the sequence $(\mu_i)_{i=0}^9$. We get 
\begin{align*}
    (  \mu_0 = 5, \mu_1 = 14, \mu_2 = 1, \mu_3 = 5, \mu_4 = 14, \mu_5 = 30, \mu_6 = 30, \mu_7 = 55, \mu_8 = 55, \mu_9 = 91).
\end{align*}
We determine the value of $l_1$: since $\sum_{i=0}^6 \mu_i = 99 < k$ and $\sum_{i=0}^7 \mu_i = 154 \ge k$, we have $l_1=7$.
\medskip

We have a sharp upper bound for $\sigma_{127}([0, 1] \times \S^{4},g)$, given by 
\begin{align*}
    \sigma_{127}([0, 1] \times \S^{4},g) < B_5^{127}(1) = \nu_7 = \sigma_{(4)}^N(A_{3/2}) \approx 6.78.
\end{align*}
One can then vary the value of $L$ in order to find a sharp upper bound $B_5^{127}(L)$, that we can represent in an axis system, see Figure \ref{fig: expl}.
\begin{figure}[H]
    \centering
    \includegraphics[width=\textwidth]{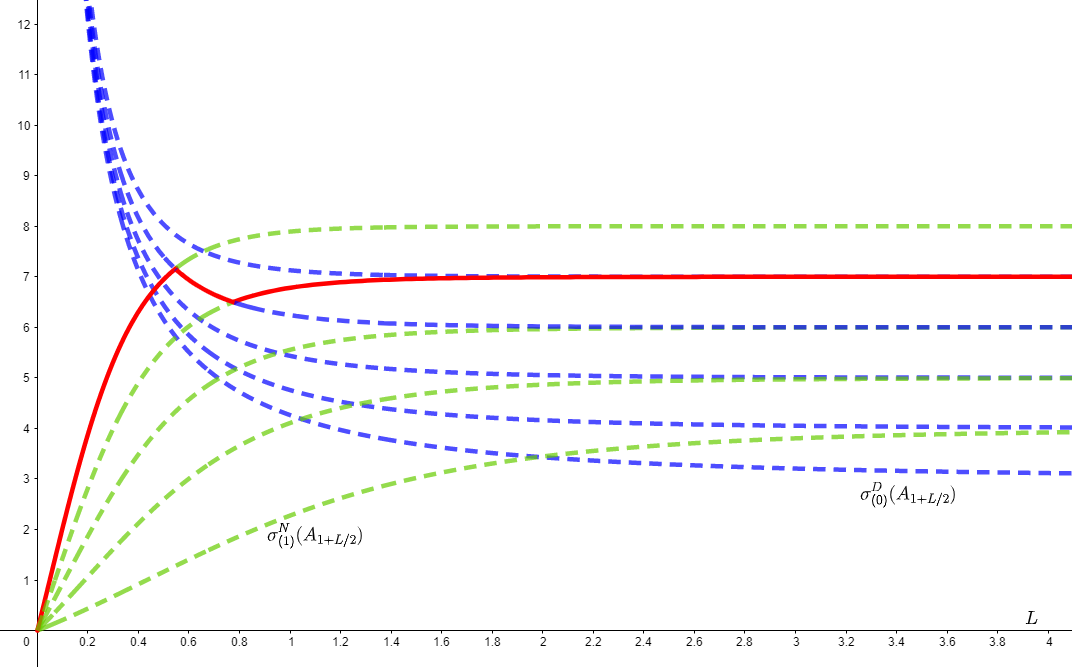}
    \caption[Sharp upper bound $B_5^{127}(L)$]{Representation of the case $n=5$. The decreasing blue curves represent  $L \longmapsto \sigma_{(i)}^D(A_{1+L/2})$ for $i = 0, \ldots, 4$ while the increasing green curves represent $L \longmapsto \sigma_{(i)}^N(A_{1+L/2})$ for $i= 1, \ldots, 5$. The solid red curve is the sharp upper bound $B_5^{127}(L)$ that we obtain thanks to the extension process.}
    \label{fig: expl}
\end{figure}

\end{expl}

Let us prove that $B_n^k(L)$ is indeed the bound that we were looking for.

\begin{proof}[Proof of Theorem \ref{thm : thm principal 3}]
We will use the same strategy that was already used in the proof of \parencite[Theorem 3]{T_revolution}. Let us start by showing that $\sigma_k(M,g) < B_n^k(L)$. 
\medskip

We consider $\nu_0, \ldots, \nu_{l_1}$ and the $k_1+1$ independent eigenfunctions $\xi_0, \ldots, \xi_{k_1}$ associated,  where $\nu_{i}$ is either a Steklov-Dirichlet or Steklov-Neumann eigenvalue on $A_{1 + L/2}$.
Note that by construction, we have $k_1 +1 \ge k$.

For each $i = 0, \ldots, k_1$ the function
\begin{align*}
    \xi_{i} : [0, \frac{L}{2}] \times \S^{n-1} \longrightarrow \R,
\end{align*}
is either a Steklov-Dirichlet or Steklov-Neumann eigenfunctions, dependending on the nature of $\nu_i$, hence we split the proof into two cases:
\begin{enumerate}
    \item Let us suppose that $\xi_{i}$ is a Steklov-Dirichlet eigenfunction. Then $\xi_{i}(L/2)=0$ and  we can extend $\xi_{i}$ to a new  function 
    \begin{align*}
        \tilde{\xi}_{i} : [0, L] \times \S^{n-1} & \longrightarrow \R \\
        (r, p) & \longmapsto \left\{ 
        \begin{array}{ll}
            \xi_{i}(r,p) & \mbox{ if } 0 \le r \le L/2  \\
            -\xi_{i}(L-r,p) & \mbox{ if } L/2 \le r \le L. 
        \end{array}
        \right.
    \end{align*}
    The function $\tilde{\xi}_{i}$ is continuous. 
    \item Let us suppose that $\xi_{i}$ is a Steklov-Neumann eigenfunction. Then we can extend $\xi_{i}$ to a new  function 
    \begin{align*}
        \tilde{\xi}_{i} : [0, L] \times \S^{n-1} & \longrightarrow \R \\
        (r, p) & \longmapsto \left\{ 
        \begin{array}{ll}
            \xi_{i}(r,p) & \mbox{ if } 0 \le r \le L/2  \\
            \xi_{i}(L-r,p) & \mbox{ if } L/2 \le r \le L. 
        \end{array}
        \right.
    \end{align*}
    The function $\tilde{\xi}_{i}$ is continuous. 
   
\end{enumerate}

One can check that all these functions $\tilde{\xi}_i$ are two by two orthogonal, therefore 
    $$\dim \, \mathrm{Span}(\xi_0, \ldots, \xi_{k_1}) = k_1 +1 \ge k.$$
Now, using the Min-Max  characterization (\ref{form : min max}) of the Steklov eigenvalues, we have

\begin{align*} 
    \sigma_k(M, g) & \le \max_{i=0, \ldots, k_1} \left\{ R_g(\tilde{\xi}_{i}) \right\}  < \max_{i=0, \ldots, k_1} \left\{ R_{\tilde{g}}(\tilde{\xi}_{i}) \right\}, \nonumber  
\end{align*}
    where  $\tilde{g}=  dr^2 + \tilde{h}^2g_0$  is given by a symmetric $\tilde{h}$ satisfying
    \begin{align*}
        \tilde{h}(r) = \begin{cases}
            1 + r & \text{if } 0 \leq  r \leq  m-1, \\
            1 + L - r & \text{if } L - m + 1 \leq r \leq L, \\
            \frac{1 + L/2 + m}{2} & \text{if } r = \frac{L}{2},
        \end{cases}
    \end{align*}
    and $h(r) > m$ for $r \in ( m-1, L-m+1 )$,
    where $m = \max_{r \in [0,L]} h(r)$. The inequality between Rayleigh quotients is then a consequence of the fact that, by construction, $\tilde{h}(r) \geq h(r)$ for all $r$ (see the proof of \parencite[Theorem 2]{T_revolution} for details).
    
    Therefore, we have
\begin{align*}
     \sigma_k(M, g) & \le \max_{i=0, \ldots, k_1}  \left\{ R_g(\tilde{\xi}_{i})  \right\} 
      <  \max_{i=0, \ldots, k_1}  \left\{ R_{\tilde{g}}(\tilde{\xi}_{i})  \right\} \\
    & = \max_{i=0, \ldots, k_1}  \left\{ \frac{\int_0^L \int_{\S^{n-1}} \left( (\partial_r \tilde{\xi}_{i})^2 + \frac{1}{\tilde{h}(r)^2} | \Tilde{\nabla} \Tilde{\xi}_{i}|^2  \right) \tilde{h}(r)^{n-1}   dV_{g_0}dr}{\int_\Sigma \tilde{\xi}_{i}^2(0,p) dV_{\Sigma}}  \right\}\\
    & = \max_{i=0, \ldots, k_1}  \left\{ \frac{2 \times \int_0^{L/2} \int_{\S^{n-1}} \left( (\partial_r \tilde{\xi}_{i})^2 + \frac{1}{\tilde{h}(r)^2} | \Tilde{\nabla} \tilde{\xi}_{i}|^2  \right) \tilde{h}(r)^{n-1}  dV_{g_0}dr}{2 \times \int_{\S^{n-1}} \tilde{\xi}_{i}^2(0,p) dV_{g_0}}  \right\}   \\
    & < \max_{i=0, \ldots, k_1}  \left\{  \frac{\int_0^{L/2} \int_{\S^{n-1}} \left( (\partial_r \xi_{i})^2 + \frac{1}{(1+r)^2} | \Tilde{\nabla}\xi_{i}|^2  \right) (1+r)^{n-1}   dV_{g_0}dr}{\int_{\S^{n-1}} \xi_{i}^2(0,p) dV_{g_0}}   \right\} \\
    & = \frac{\int_0^{L/2} \int_{\S^{n-1}} \left( (\partial_r \xi_{k_1})^2 + \frac{1}{(1+r)^2} | \Tilde{\nabla}\xi_{k_1}|^2  \right) (1+r)^{n-1}   dV_{g_0}dr}{\int_{\S^{n-1}} \xi_{k_1}^2(0,p) dV_{g_0}}   \quad  (\mbox{by construction}) \\
    & = \nu_{l_1}, 
\end{align*}
where we use the symmetry of $\tilde{g}$ on the third line and the second strict inequality comes from the  existence of a continuum of points $r \in [0, L/2]$ such that $\tilde{h}(r) < 1+r$. 
\medskip

Let $B : = \nu_{l_1}$, we will prove that $B$ is the desired sharp upper bound $B_n^k(L)$. It remains to show that $B$  is sharp, i.e for all $\epsilon >0$, there exists a metric of revolution $g_\epsilon$ on $M$ such that 
\begin{align*}
    \sigma_k(M, g_\epsilon) > B -\epsilon.
\end{align*}

Let $\epsilon >0$. 
Let $M=[0, L] \times \S^{n-1}$ and let $g_\epsilon(r,p) = dr^2+h_\epsilon^2(r)g_0(p)$ be a metric of revolution on $M$ such that 
\begin{enumerate}
    \item  For all $r \in [0, L]$, we have $h_\epsilon(r)=h_\epsilon(L-r)$ (i.e $h_\epsilon$ is symmetric).
    \item For all $r \in [0, L/2- \delta]$, we have $h_\epsilon(r)=(1+r)$, with $\delta$ small enough to guarantee that for all $r \in [0, L/2]$, we have $\max\{(1+r)^{n-3}-h_\epsilon(r)^{n-3}, (1+r)^{n-1}-h_\epsilon(r)^{n-1} \}  < \frac{\epsilon}{B} =: \epsilon^*$. 
\end{enumerate}

We  write $f_k$ a Steklov eigenfunction associated with $\sigma_k(M, g_\epsilon)$. Because $g_\epsilon$ is symmetric, we can choose $f_k$ symmetric or anti-symmetric. Writing 
\begin{align*}
    R_{A_{1+L/2}}(f_k) :=  \frac{ \int_0^{{L}/2} \int_{\S^{n-1}} \left( (\partial_r f_k)^2 + \frac{1}{(1+r)^2} | \Tilde{\nabla} f_k|^2  \right) (1+r)^{n-1}   dV_{g_0}dr}{ \int_{\S^{n-1}} (f_k)^2(0,p) dV_{g_0}},
\end{align*}
it is an easy computation to check that 
\begin{align*}
     R_{A_{1+L/2}}(f_k) < \sigma_k(M, g_\epsilon) + \epsilon.
\end{align*}
We once again split the proof into two cases:
\begin{enumerate}
    \item Let us suppose that $f_k$ is  anti-symmetric, i.e we have $f_k(r,p)=u_k(r)S_l(p)$ with $u_k$ anti-symmetric and $B = \sigma_j^D(A_{1+L/2})$ for a certain $j \le l$. Then we can check that 
    \begin{align*}
        f_k(L/2, p) = 0 \mbox{ and } \int_{\{0\} \times \S^{n-1}} \varphi_i f_k dV_{g_0} = 0 \mbox{ for all } i= 0, \ldots, l-1.
    \end{align*}
    \item Let us suppose that $f_k$ is symmetric, i.e we have $f_k(r,p)=u_k(r)S_l(p)$ with $u_k$ symmetric  and $B = \sigma_j^N(A_{1+L/2})$ for a certain $j \le l$. Then we can check that 
    \begin{align*}
        \partial_r f_k(L/2, p) =0 \mbox{ and }  \int_{\{0\} \times \S^{n-1}} \phi_i f_k dV_{g_0} = 0 \mbox{ for all } i= 0, \ldots, l-1.
    \end{align*}
     \end{enumerate}
    In both cases,  we can use $f_k|_{[0, L/2] \times \S^{n-1}}$ as a test function for $B$, and we get 
    \begin{align*}
         B < R_{A_{1+L/2}}(f_k) < \sigma_k(M, g_\epsilon) + \epsilon.
    \end{align*}

Therefore, $B = \nu_{l_1}$ is indeed the sharp upper bound $B_n^k(L)$.

\end{proof}

From this statement, let us prove Corollary \ref{cor : du principal 3}.
\begin{proof}
Let $k \ge 1$. We want to show the existence of a bound $B_n^k < \infty$ such that for all hypersurfaces of revolution $(M, g)$ of dimension $n \ge 3$, we have $\sigma_k(M,g) < B_n^k$.  A way to do it is the analyse to function $ L \longmapsto B_n^k(L)$, and defining $B_n^k$ as follows:
\begin{align*}
    B_n^k := \sup \{B_n^k(L) : L \in \R_+^* \}.
\end{align*}
By construction, we have $\sigma_k(M,g) < B_n^k$. We only have to show that $B_n^k$ is finite. Indeed, for all $L \in \R_+^*$, we have
\begin{align*}
    B_n^k(L) < \sigma_{(k)}^N(L).
\end{align*}
Therefore, we have 
\begin{align*}
    B_n^k & = \sup \{B_n^k(L) : L \in \R_+^* \} \\
    & \le \sup \{ \sigma_{(k)}^N(A_{1+L/2}) : L \in \R_+^* \} \\
    & = \lim_{L \to \infty} \sigma_{(k)}^N(A_{1+L/2}) \\
    & = n + k -2.
\end{align*}
Now we can prove that $B_n^k$ is a sharp upper bound, that is for all $\epsilon >0$, there exists a hypersurface of revolution $(M_\epsilon, g_\epsilon)$ such that $\sigma_k(M_\epsilon, g_\epsilon) > B_n^k - \epsilon$. 
\medskip

Let $\epsilon >0$. There exists $L^* \in \R_+^*$ such that $B_n^k - B_n^k(L^*) < \frac{\epsilon}{2}$. We define $M_\epsilon := [0, L^*] \times \S^{n-1}$. Thanks to Theorem \ref{thm : thm principal 3}, there exists a metric of revolution $g_\epsilon$ on $M_\epsilon$ such that 
\begin{align*}
    \sigma_k(M_\epsilon, g_\epsilon) > B_n^k(L^*) - \frac{\epsilon}{2} > B_n^k- \frac{\epsilon}{2}- \frac{\epsilon}{2}= B_n^k-\epsilon.
\end{align*}
\end{proof}

From Theorem \ref{thm : thm principal 3}, let us prove Corollary \ref{cor : deuxieme cor du principal 3}.
\begin{proof}
Let us notice that given $j \in \N$ fixed, we have 
\begin{align*}
    \lim_{L\to 0}\sigma_{(j)}^D(A_{1+L/2}) = \infty \quad \mbox{ and } \quad
    \lim_{L\to 0} \sigma_{(j)}^N(A_{1+L/2}) = 0.
\end{align*}
Let $k \in \N$ and let $\epsilon >0$. Let $L_0 >0$ be small enough so that $\sigma_{(1)}^N(A_{1+L_0/2}) < \ldots < \sigma_{(k)}^N(A_{1+L_0/2}) < \epsilon$. Let $0<L< L_0$. Then the finite set $E=E(n,k,L)$ can be ordered and $\nu_0, \ldots, \nu_{l_0+1}$ are equal respectively to $\sigma_{(1)}^N(A_{1+L/2}), \ldots, \sigma_{(l_0+1)}^N(A_{1+L/2})$, where $l_0+1<k$. 

Hence $\nu_{l_1} \in \{ \sigma_{(1)}^N(A_{1+L/2}), \ldots, \sigma_{(l_0+1)}^N(A_{1+L/2}) \}$ and therefore, writing $M:=[0, L]\times \S^{n-1}$,
\begin{align*}
    \sigma_k(M,g) < \nu_{l_1} \le \max\{ \sigma_{(1)}^N(A_{1+L/2}), \ldots, \sigma_{(l_0+1)}^N(A_{1+L/2}) \} < \epsilon.
\end{align*}
\end{proof}

\section{Question and conjecture} \label{sect : conj}

We now investigate if the sharp upper bound $B_n^k$ is achieved by a finite critical length $L_k$ or an infinite one. In particular, we are interested to know if the infinite critical length happens only finitely often. As the question is complex, we used a program that, given a dimension $n$ and an integer $k$, checks if the eigenvalues with index smaller or equal to $k$ have finite or infinite critical length. The code for this program is given in Appendix \ref{appendix : python codes}.
\medskip

Such a program requires quite some computation time to run, at least if we want to check many eigenvalues.
Here are some considerations used in order to optimize the program.
\begin{defn}
    Given a dimension $n \ge 3$, we say that $k \in \N$ is a \emph{diagnosis eigenvalue} if $k =  \sum_{j=0}^{i-1}2m_j$ for a certain integer $i$.
\end{defn}

This definition is motivated by the following lemma:
\begin{lemma}
Let $n \ge 3$ and $i \in \N$. Let $k := \sum_{j=0}^{i-1}2m_j$ be the $i$th diagnosis eigenvalue. Let us suppose that $k$ has an associated finite critical length. Then, for each $k' \in \{k, \dots, k + 2m_i-1\}$, the eigenvalue $k'$ has an associated finite critical length.
\end{lemma}

\begin{proof}
    Since $k = \sum_{j=0}^{i-1}2m_j$, there exists $C > 0$ such that $B_n^k(L) = \sigma_{(i)}^N(A_{1+L/2})$ for all $L >C$. Moreover, since $L \longmapsto \sigma_{(i)}^N(A_{1+L/2})$ is strictly increasing, then the critical length $L_k$ (which is finite by assumption) satisfies $L_k  \le C$. 
    \medskip
    
    Let us fix $L^* >C$. Then for each $k'$ such that $k \le k' < k + m_i$, we have $B_n^{k'}(L^*) = \sigma_{(i)}^N(A_{1+L^*/2})$. Moreover, we have $B_n^{k'}(L_k) \ge B_n^k(L_k)$. Thus, we have 
   \begin{align*}
       \sigma_{(i)}^N(A_{1+L^*/2}) <  B_n^k(L_k) \le B_n^{k'}(L_k)  \le  \sup_{L \in (0, C]} B_n^{k'}(L),
   \end{align*}
   and $k'$ has a finite critical length. 
    \medskip
    
   Finally, for $k'$ such that
   \begin{align*}
       k+ m_i \le k' < k+2m_i,
   \end{align*}
   we have $B_n^{k'}(L) = \sigma_{(i)}^D(A_{1+L/2})$ for all $L >C$. Moreover, since $L \longmapsto \sigma_{(i)}^D(A_{1+L/2})$ is decreasing, we know that $k'$ has a finite critical length. 
    \medskip
    
   Therefore, for all $k' \in \{k, \dots, k + 2m_i-1\}$, the eigenvalue $k'$ has an associated finite critical length.
   
\end{proof}

\begin{expl}
Let $n=5$. The sequence of diagnosis eigenvalues is
\begin{align*}
    (2, 12, 40, 100, 210, 392, 672, 1'080, 1'650, 2'420, 3'432, 4'732, 6'370, 8'400, 10'880,  \ldots).
\end{align*}
The lemma allows us to state that in dimension $5$, \textit{if} the $14$th diagnosis eigenvalue $k=8'400$ has an associated finite critical length, \textit{then} for each $k' \in \{8'400, \dots, 10'879\}$, the $k'$th eigenvalue has an associated finite critical length.
\end{expl}

Therefore, it is relevant to consider only these diagnosis eigenvalues to save time while checking what could be the answer to our open question.
\medskip

Here are some of the results obtained by our program:
\begin{align*}
    \begin{array}{ccl}
     n & i & \mbox{Result of the program's investigations}  \\
     \hline
     3 & 150 & \mbox{From $k=18$ to $k=45'601$, only finite critical lengths found.}\\
     4 &  40 & \mbox{From $k=408$ to $k=47'641$, only finite critical lengths found.} \\
     5 & 30 & \mbox{From $k=8'400$ to $k=195'423$, only finite critical lengths found.} \\
     6 & 25 & \mbox{From $k=21'112$ to $k=610'973$, only finite critical lengths found.}
\end{array}
\end{align*}
We draw attention to the fact that the program \textit{does not say} that in dimension $3$, the eigenvalue number $17$ has a critical length at infinity, it only says that it has not found any infinite critical length above $17$.
Actually, in dimension $n=3$, we have not found any infinite critical length associated with the eigenvalue $k$, if $k >8$.
\medskip

These numerical results motivate \Cref{conj : critical length}:

\begin{center}
    \textit{Let $n \ge 3$ be an integer. Then there exists $K=K(n) \in \N$ such that for every $k \ge K$, the $k$th eigenvalue has an associated finite critical length.}
\end{center}

If this conjecture were to be proven, it would then be interesting to study the function $$n \longmapsto K_{min}(n),$$ where $K_{min}(n)$ is the smallest integer in the conjecture. From our numerical experiments, it seems that this function, \textit{if it exists}, grows quickly with $n$.

\section{Proof of the conjecture in low dimension} \label{sect : low dimension}

\textbf{Notation.} If the context is clear, we write $\sigma_{(\kappa)}^{D/N}(L)$ instead of $\sigma_{(\kappa)}^{D/N}(A_{1+L/2})$ to streamline the notations.
\medskip

The goal of this section is to prove \Cref{thm : low dimension}. 
The idea of the proof is simple: as explained in previous sections, the bound $B_n^k(L)$ can be obtained by looking at Steklov-Neumann $\sigma^N$ and Steklov-Dirichlet $\sigma^D$ eigenvalues on an annulus. So suppose that for a given $k$, the bound $B_n^k(L)$ is achieved by some $\sigma_{(\kappa)}^N(L)$ for $L$ near 0.
We know the curve $L \mapsto \sigma_{(\kappa)}^N(L)$ will intersect the curves of $\sigma_{(0)}^D(L), \sigma_{(1)}^D(L), \dots, \sigma_{(\kappa-1)}^D(L)$. 
Hence looking at the multiplicities carried by each curve, one can see that for large enough $L$, the bound $B_n^k(L)$ cannot continue to be achieved by $\sigma_{(\kappa)}^N(L)$ and must be achieved by a lower curve.
Applying this reasoning to more curves $\sigma^N(L)$, we will obtain an upper bound on the index of the curve $\sigma^N$ which can be achieved for large $L$ (c.f. Lemma \ref{lem : tech 2}).
This allows us to obtain an upper bound for $B_n^k(L)$ when $L \to \infty$.
Meanwhile, looking at the intersection of $\sigma_{(\kappa)}^N(L)$ and $\sigma_{(0)}^D(L)$, we obtain a lower bound for $B_k^n$ (Lemma \ref{lem : tech 1}).
In dimensions $n = 3$ or $n = 4$, this bound is bigger than the one achieved at infinity allowing us to conclude that we must have a finite critical length. 
\medskip


We recall that the multiplicity $m_k$ of the eigenvalue $\sigma_{(k)}^{N/D}(A_{1+L/2})$ is given by 
\begin{align*}
   m_k &= \frac{(n + k -3)!}{(n-2)! k!} (n+2k - 2) \\
   &= \frac{(k + n - 3)(k + n - 4) \dots (k + 2)(k + 1)}{(n-2)!}(2k + n - 2).
\end{align*}

We first prove the following lemma:

\begin{lemma} \label{lem : tech 1}
    Let $n \in \N$. For $\kappa \in \N$, let $b_\kappa \in \R$ be the value such that
    \begin{align*}
        \sigma_{(\kappa)}^N(A_{1+L/2}) = b_\kappa = \sigma_0^D(A_{1+L/2}).
    \end{align*}
    Let $c_0 \approx 0.83$ be the unique positive solution of the equation $\frac{1+c}{1-c} = e^{2/c}$. 
    Then for any $0< c < c_0$, there exists $\kappa_0(c)$ such that for all $\kappa \geq \kappa_0$, 
    \begin{align*}
        b_\kappa \geq c\kappa.
    \end{align*}
\end{lemma}

\begin{proof}
    Since $\sigma_0^D$ and $\sigma_{(\kappa)}^N$ are monotone in $L$, we can invert them and give an expression for $(\sigma_0^D)^{-1}$ and $(\sigma_{(\kappa)}^N)^{-1}$. We get
    \begin{align*}
        (\sigma^D_0)^{-1}(y) = \left(1 + \frac{n-2}{y - (n-2)}\right)^{\frac{1}{n-2}} 
    \end{align*}
    and
    \begin{align*}
        (\sigma^N_{(\kappa)})^{-1}(y) = \left(\frac{1 + \frac{y}{\kappa}}{1 - \frac{y}{\kappa + n - 2}}\right)^{\frac{1}{2\kappa + n -2}}.
    \end{align*}
    With the notation $y_\kappa = c \kappa$, we have to show that 
    \begin{align*}
        (\sigma^N_{(\kappa)})^{-1}(y_\kappa) \leq (\sigma^D_0)^{-1}(y_\kappa)
    \end{align*}
    if $\kappa$ is large enough. Here we use that $\sigma^N(\kappa)$ is increasing and $\sigma^D_0$ is decreasing. 
    \medskip
    
    Therefore, we need to show that for large enough $\kappa$, we have
    \begin{align*}
        \frac{1 + \frac{y_\kappa}{\kappa}}{1 - \frac{y_\kappa}{\kappa + n -2}} \leq \left(1+ \frac{n - 2}{y_\kappa - (n - 2)}\right)^{\frac{2\kappa + n - 2}{n - 2}}.
    \end{align*}

    We have for the left-hand-side
    \begin{align*}
        \frac{1 + \frac{c\kappa}{\kappa}}{1 - \frac{c\kappa}{\kappa + n - 2}} = \frac{1 + c}{1-c} + \mathcal{O}(\frac{1}{\kappa}),
    \end{align*}
    and for the right-hand-side
    \begin{align*}
        \left(1 + \frac{n - 2}{c\kappa - (n - 2)}\right)^{\frac{2\kappa + n - 2}{n - 2}}
        &= e^{\frac{2\kappa + n - 2}{n - 2} \log\left(1 + \frac{n-2}{c\kappa - (n - 2)}\right)} \\
        &= e^{\frac{2\kappa + n - 2}{n - 2} \left(\frac{n - 2}{c\kappa - (n - 2)} + \mathcal{O}(\frac{1}{\kappa^2})\right)} \\
        &= e^{\frac{2}{c} + \mathcal{O}(\frac{1}{\kappa})} \\
        &= e^{\frac{2}{c}} + \mathcal{O}(\frac{1}{\kappa}).
    \end{align*}
    Therefore,  we have 
    \begin{align*}
         \lim_{\kappa \to \infty} \frac{1 + \frac{c\kappa}{\kappa}}{1 - \frac{c\kappa}{\kappa + n - 2}} = \frac{1+c}{1-c} < e^{\frac{2}{c}} = \lim_{\kappa \to \infty} \left(1 + \frac{n - 2}{c\kappa - (n - 2)}\right)^{\frac{2\kappa + n - 2}{n - 2}},
    \end{align*}
    since we assumed $0 < c < c_0$.
    
\end{proof}

\begin{rem}
    This lemma holds whatever the dimension $n \ge 3$. Moreover, the value of $c_0 \approx 0.83$ is independent of the dimension.
\end{rem}

\begin{lemma} \label{lem : tech 2}
    Let $c_0 \approx 0.83$ be the unique positive solution of the equation $\frac{1+c}{1-c} = e^{2/c}$. For $n = 3$ or $n = 4$, there exists $\kappa_1 \in \mathbb{N}$ such that  for all $\kappa \geq \kappa_1$,
    \begin{align*}
        m_\kappa + m_{\kappa-1} + \dots + m_{\lfloor c_0\kappa \rfloor} < m_0 + m_1 + \dots + m_{\lfloor c_0 \kappa \rfloor - 1}.
    \end{align*}
\end{lemma}

\begin{proof}
    Let us suppose  $n=3$. Then we have $m_i = 2i+1$, therefore
    \begin{align*}
        m_0 + m_1 + \ldots + m_\kappa = 1 + 3 + \ldots + (2\kappa +1) = (\kappa+1)^2,
    \end{align*}
    and
    \begin{align*}
        m_0 + m_1 + \ldots + m_{\lfloor c_0\kappa \rfloor -1} = \lfloor c_0\kappa \rfloor^2.
    \end{align*}
    Hence, we have
    \begin{align*}
        \frac{ m_0 + \ldots + m_\kappa }{ m_0 + m_1 + \ldots + m_{\lfloor c_0\kappa \rfloor -1}} = \frac{(\kappa+1)^2}{ \lfloor c_0\kappa \rfloor^2} \underset{\kappa \to \infty}{\longrightarrow} \frac{1}{c_0^2} < 2.
    \end{align*}
    Therefore, there exists $\kappa_1 \in \N$ such that for all $\kappa \ge \kappa_1$, we have 
    \begin{align*}
        m_\kappa + m_{\kappa-1} + \dots + m_{\lfloor c_0\kappa \rfloor} < m_0 + m_1 + \dots + m_{\lfloor c_0 \kappa \rfloor - 1}.
    \end{align*}
    
    Let us suppose $n=4$. Then we have $m_i = (i+1)^2$, therefore 
    \begin{align*}
        m_0 + m_1 + \ldots + m_\kappa = \frac{(\kappa+1)(\kappa+2)(2\kappa+3)}{6}
    \end{align*}
    and 
    \begin{align*}
         m_0 + m_1 + \ldots + m_{\lfloor c_0\kappa \rfloor -1} = \frac{\lfloor c_0\kappa \rfloor (\lfloor c_0\kappa \rfloor+1) (2\lfloor c_0\kappa \rfloor+1)}{6}.
    \end{align*}
    Hence, we have
    \begin{align*}
        \frac{ m_0 + m_1 + \ldots + m_\kappa }{ m_0 + m_1 + \ldots + m_{\lfloor c_0\kappa \rfloor -1}}  \underset{\kappa \to \infty}{\longrightarrow} \frac{1}{c_0^3} < 2.
    \end{align*}
    Therefore, there exists $\kappa_1 \in \N$ such that for all $\kappa \ge \kappa_1$, we have 
    \begin{align*}
        m_K + m_{\kappa-1} + \dots + m_{\lfloor c_0\kappa \rfloor} < m_0 + m_1 + \dots + m_{\lfloor c_0 \kappa \rfloor - 1}.
    \end{align*}
\end{proof}

\begin{rem}
    In dimension $n \geq 5$, the previous lemma does not hold. Indeed, one can see that $m_i$ is polynomial of degree $(n-2)$, $m_i = \frac{2i^{n-2}}{(n-2)!} + \mathcal{O}(i^{n-3})$. Hence, if we search for the best $c$ such that the lemma hold, we must solve for $c$ and in the limit $\kappa \to \infty$, the equation
    \begin{align*}
        \frac{2 \kappa^{n-1}}{(n-1)(n-2)!} + \mathcal{O}(\kappa^{n-2}) = \sum_{i = 0}^\kappa m_i = \sum_{i = 0}^{c\kappa} m_i = \frac{2 (c\kappa)^{n-1}}{(n-1)(n-2)!} + \mathcal{O}(\kappa^{n-2}).
    \end{align*}
    This gives $c = \left(\frac{1}{2}\right)^{\frac{1}{n}}$. As $n \to \infty$, $c \to 1$ and in fact for $n \geq 5$, $c > c_0$. 
    
\end{rem}


\begin{lemma}
    For $n \in \{3, 4\}$, there exists a $K_3 \in \mathbb{N}$ such that for any $k \geq K_3$, we have 
    \begin{align*}
        B_n^k = \sup_{L \in (0, \infty)} \{B_n^k(L) \} > \lim_{L \to \infty} B_n^k(L).
    \end{align*}
\end{lemma}

\begin{proof}
    Let $n \in \{3, 4\}$. Let $\kappa \ge K_3 := \max\{\kappa_0, \kappa_1\}$, where $\kappa_0$ and $\kappa_1$ are as in the preceding lemmas. Let $k$ be such that 
    \begin{align*}
        m_0 + m_1 + \ldots + m_{\kappa-1} < k \le m_0 + m_1 + \ldots m_\kappa.
    \end{align*}
    Then, for $L$ small enough, we have 
    \begin{align*}
        B_n^k(L) = \sigma_\kappa^N(A_{1+L/2}).
    \end{align*}
    Moreover,
    \begin{align*}
        k \le
        m_0 + \ldots + m_{\lfloor 0.8 \kappa \rfloor - 1} +m_{\lfloor 0.8 \kappa \rfloor} + \ldots + m_\kappa  
        \stackrel{\Cref{lem : tech 2}}{<}
        2 \cdot \left(  m_0 + \ldots + m_{\lfloor 0.8 \kappa \rfloor - 1} \right).
    \end{align*}
    Therefore, for $L$ large enough, we have 
    \begin{align*}
        B_n^k(L) \le \sigma_{\lfloor 0.8 \kappa \rfloor -1}^N(A_{1+L/2}) \underset{L \to \infty}{\longrightarrow} \lfloor 0.8 \kappa \rfloor + n -3,
    \end{align*}
    as shown in \Cref{fig:proof low dim}.
    \begin{figure}[H]
        \centering
        \includegraphics[width=\textwidth]{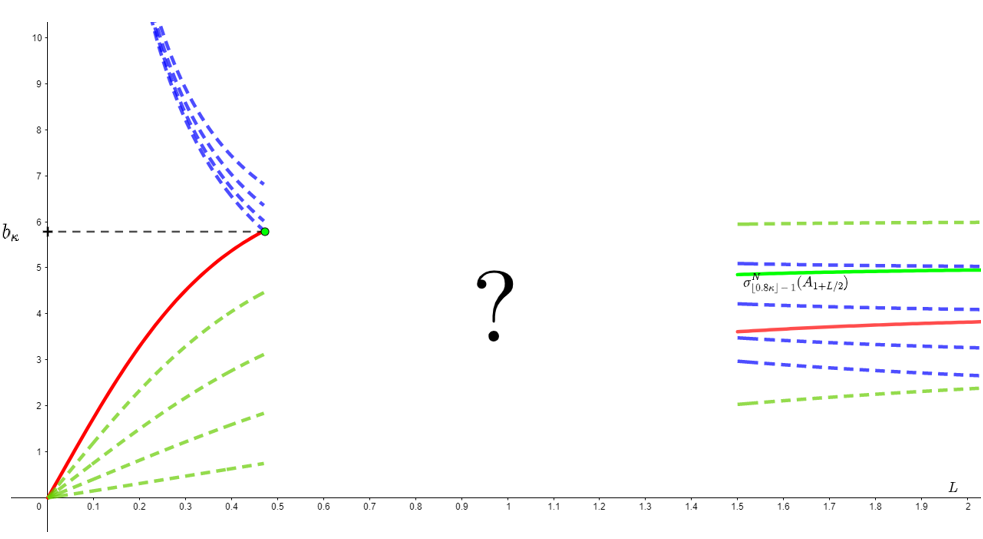}
        \caption{Before the intersections (the region covered by the interrogation mark), the sharp upper bound is given by $\sigma_{(\kappa)}^N(A_{1+L/2})$ and after, an upper bound (non necessarily sharp) is given by $\sigma_{(\lfloor 0.8 \kappa \rfloor -1)}^N(A_{1+L/2})$, which is lower than $b_\kappa$.  }
        \label{fig:proof low dim}
    \end{figure}
    Moreover, \Cref{lem : tech 1} allows us to state that $b_\kappa \ge c_0 \kappa > \lfloor 0.8 \kappa \rfloor + n -3$ because $\kappa$ is large enough. Therefore,
    \begin{align*}
        B_n^k \ge b_\kappa \ge c_0 \kappa > \lfloor 0.8 \kappa \rfloor +n - 3,
    \end{align*}
    and hence 
    \begin{align*}
        B_n^k = B_n^k(L_k) 
    \end{align*}
    for a certain $L_k$ finite, which means that $k$ has a finite critical length.
    
\end{proof}


\textbf{Acknowledgment.} The first author thanks Bruno Colbois and Katie Gittins for organizing the Neuchâtel ``Geometric Spectral Theory" meeting, in which he learned of the problem from Léonard Tschanz. He also acknowledge  support of EPSRC grant EP/T030577/1.
The second author would like to warmly thank his thesis supervisor Bruno Colbois for letting him work on this topic. Moreover, he would like to thank Maxime Welcklen for his help on the use of Python when coding the functions used in the extension process. He would also like to thank Prof. Pascal Felber who greatly improved the computing time of the codes, allowing us to search way further in a decent time, as well as Prof. Katie Gittins for her careful proofreading of a first version of the paper.

\printbibliography

\begin{appendices}

\section{Plotting sharp upper bounds} \label{sect : plot simple}
We can implement the extension process in a computer and let the value of $L$ vary in order to plot the sharp upper bound as a function of $L$, where we see $n$ and $k$ as parameters.

Here are some figures obtained. On the bottom-right part of each graphic, one can see written "$B_n^k$", meaning that we studied the  $n$th dimension and the $k$th eigenvalue. Moreover, for each values of $n$ and $k$, the graphic indicates if we found a finite or infinite critical length, and provides an estimation of the critical length as well as an estimation of  the sharp upper bound $B_n^k$.

 \begin{figure}[H]
    \centering
    \includegraphics[scale=0.455]{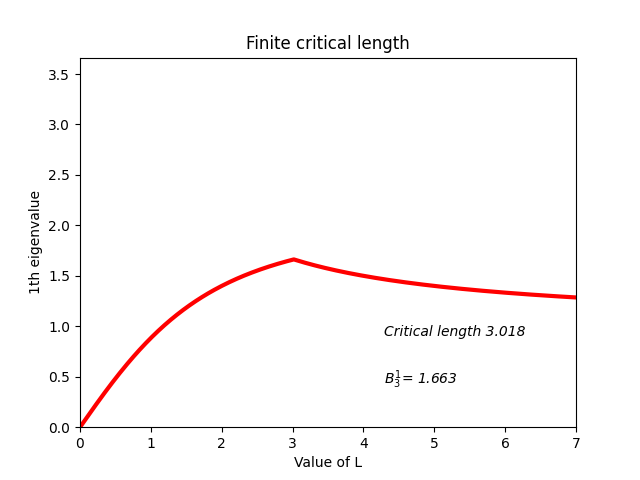}
    \includegraphics[scale=0.455]{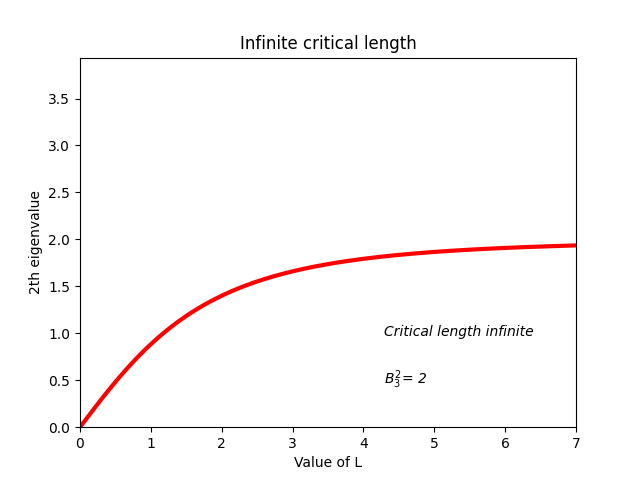}
\end{figure}
\begin{figure}[H]
    \centering
    \includegraphics[scale=0.455]{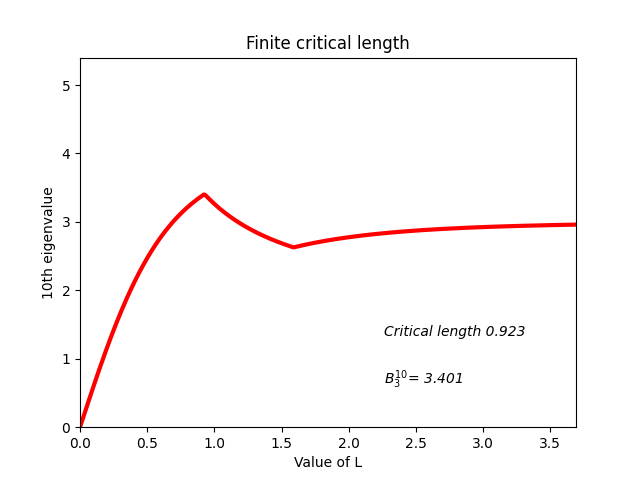}
    \includegraphics[scale=0.45]{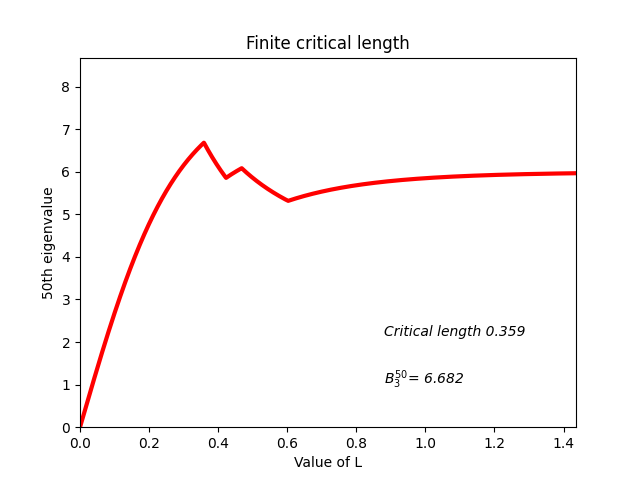}
\end{figure}
\begin{figure}[H]
    \centering
    \includegraphics[scale=0.45]{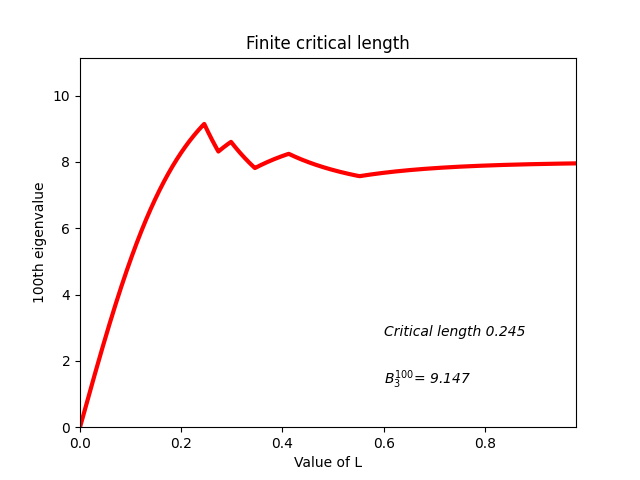}
    \includegraphics[scale=0.45]{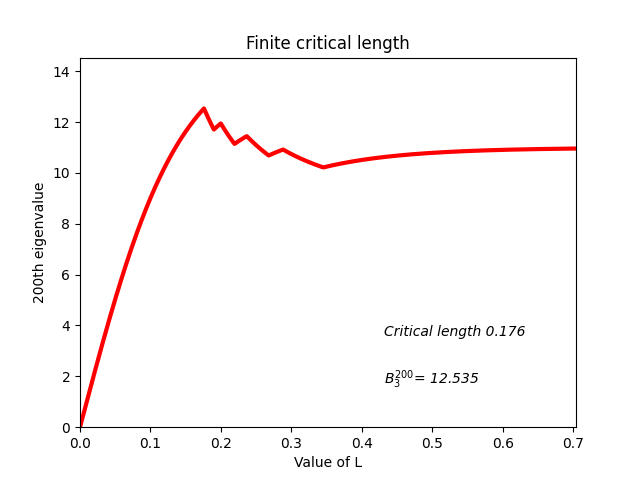}
\end{figure}

\begin{figure}[H]
     \centering
    \includegraphics[scale=0.455]{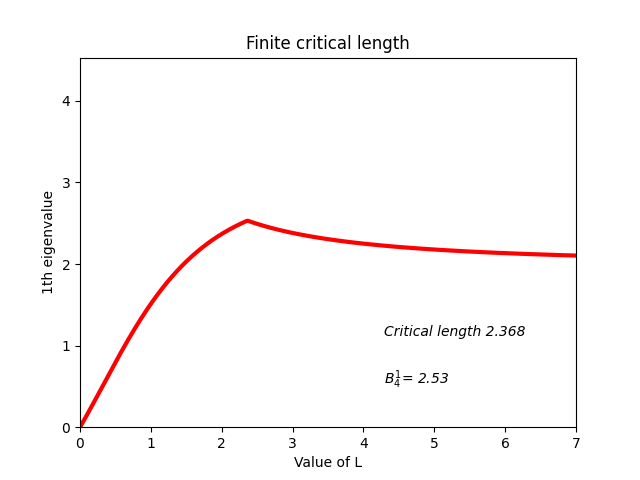}
    \includegraphics[scale=0.455]{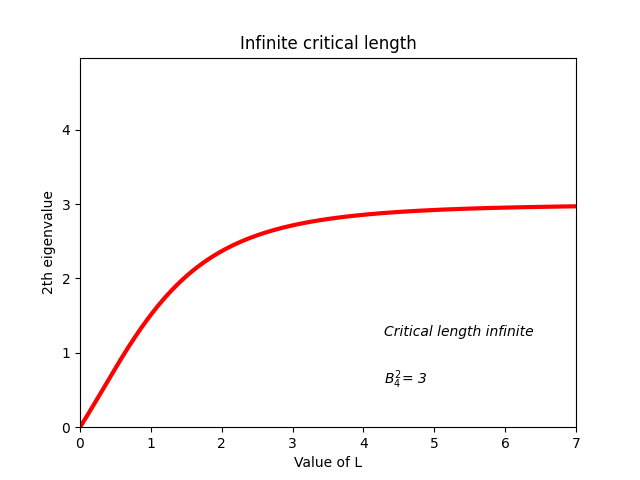}
\end{figure}
\begin{figure}[H]
    \centering
    \includegraphics[scale=0.455]{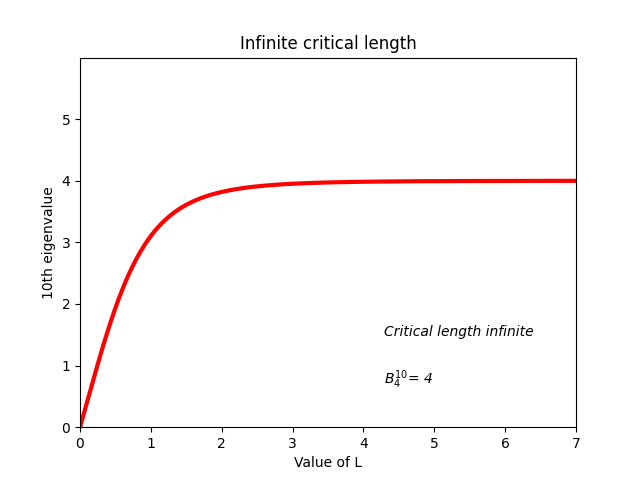}
    \includegraphics[scale=0.45]{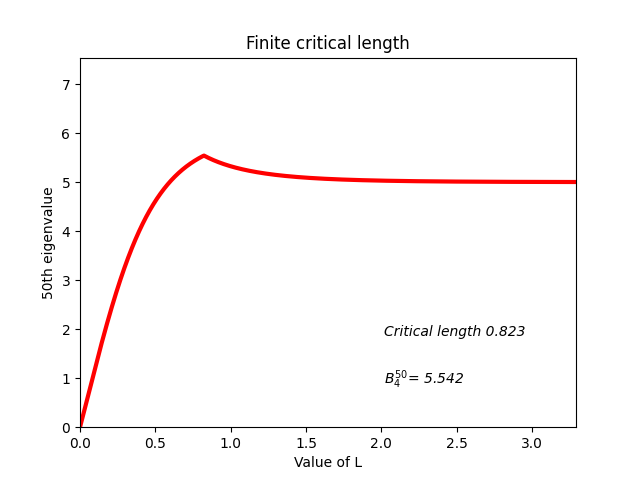}
\end{figure}
\begin{figure}[H]
    \centering
    \includegraphics[scale=0.45]{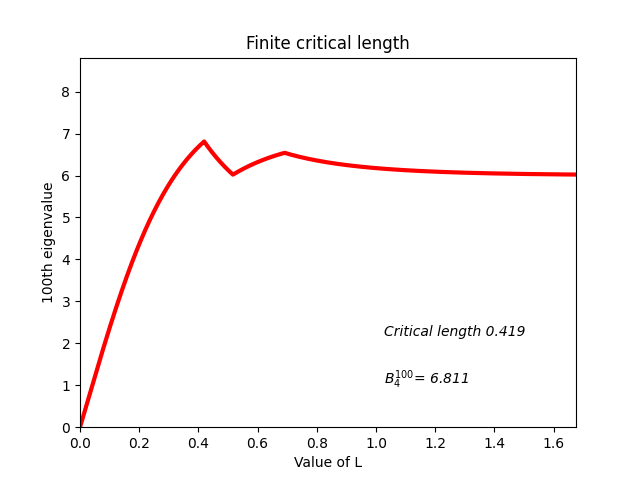}
    \includegraphics[scale=0.45]{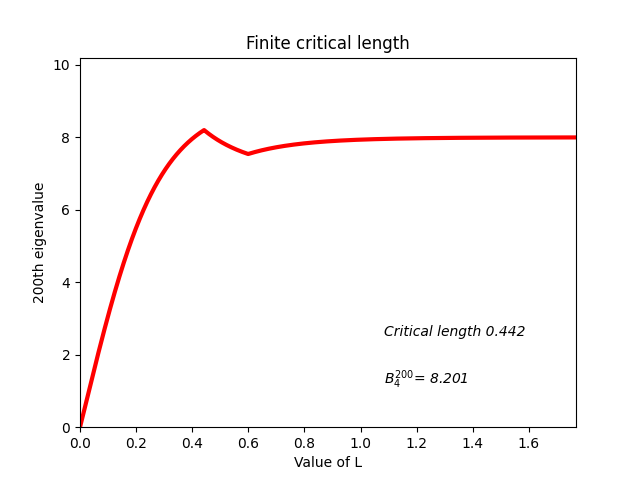}
\end{figure}

\begin{figure}[H]
     \centering
    \includegraphics[scale=0.455]{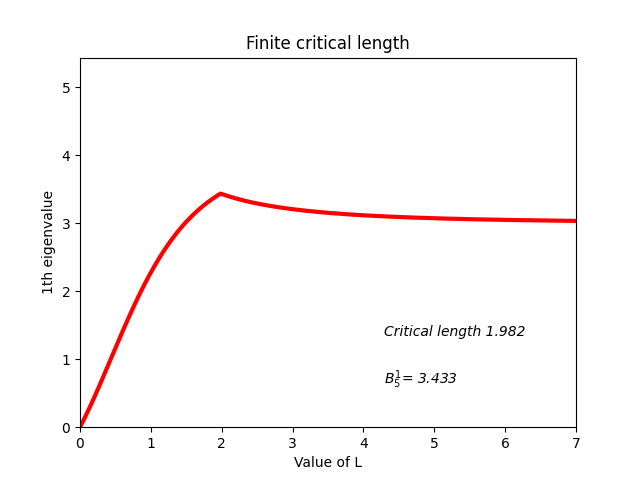}
    \includegraphics[scale=0.455]{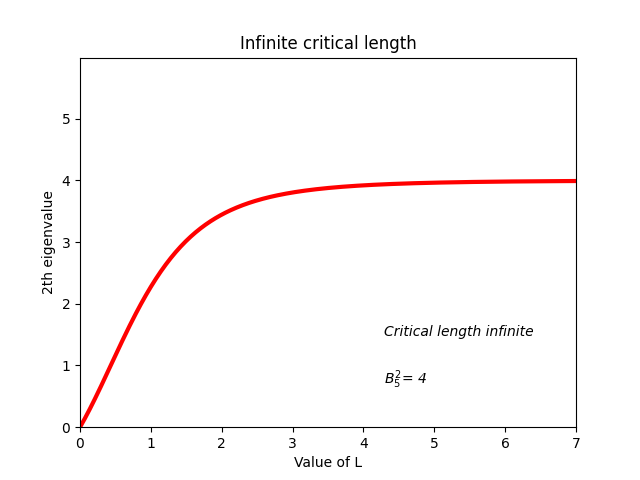}
\end{figure}
\begin{figure}[H]
    \centering
    \includegraphics[scale=0.455]{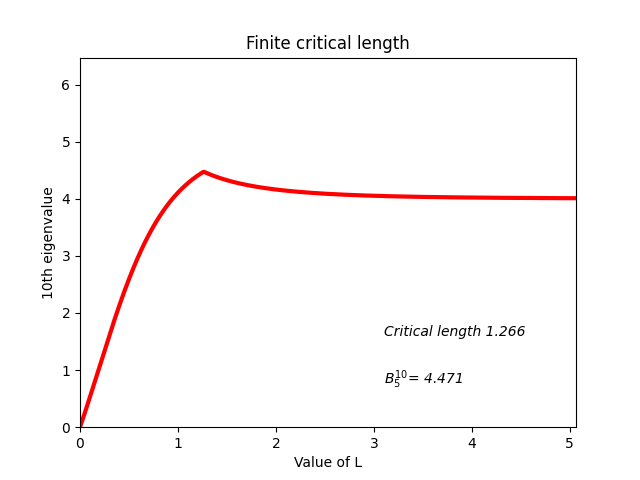}
    \includegraphics[scale=0.45]{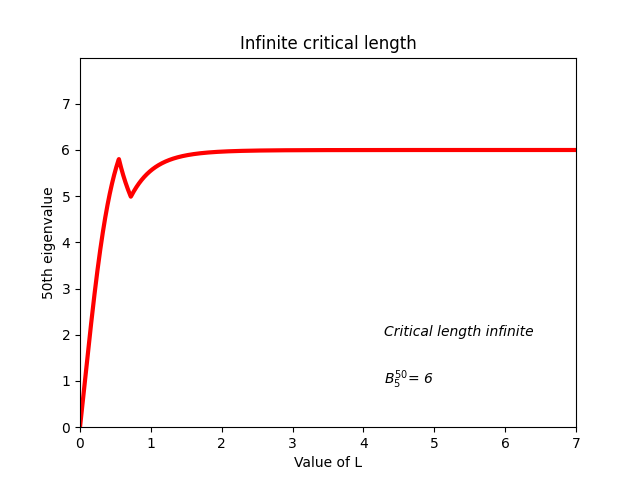}
\end{figure}
\begin{figure}[H]
    \centering
    \includegraphics[scale=0.45]{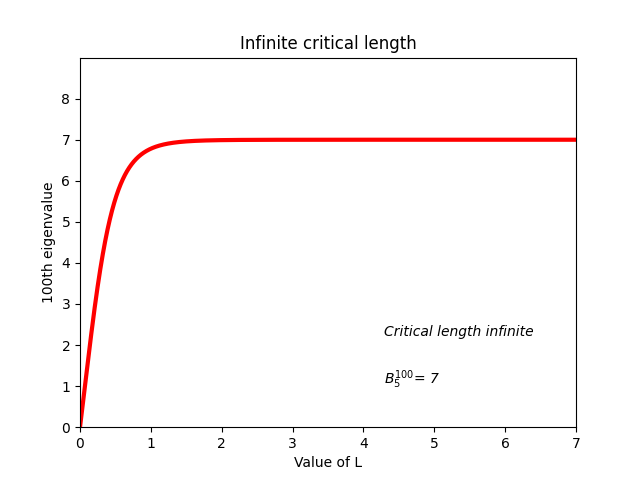}
    \includegraphics[scale=0.45]{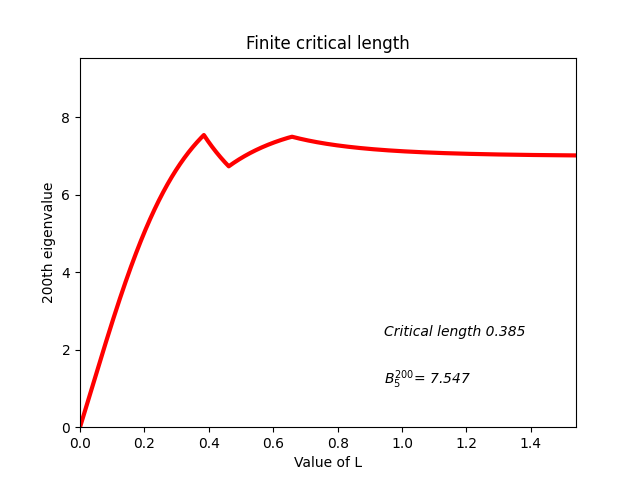}
     \caption[Different sharp upper bounds]{As one can see, the behaviour of the function $L \longmapsto B_n^k(L)$ is hard to predict, and its properties (smoothness, infinite critical length, finite critical length) depend on the values of $n$ and $k$.}
\end{figure}

\section{Adding the mixed eigenvalues} \label{sect : plot mixte}
Since the function $L\longmapsto B_n^k(L)$ comes from the mixed Steklov-Dirichlet and Steklov-Neumann eigenvalues, we can add them in the graphics to understand the function better.

We got the following figures.

\begin{figure}[H]
     \centering
    \includegraphics[scale=0.455]{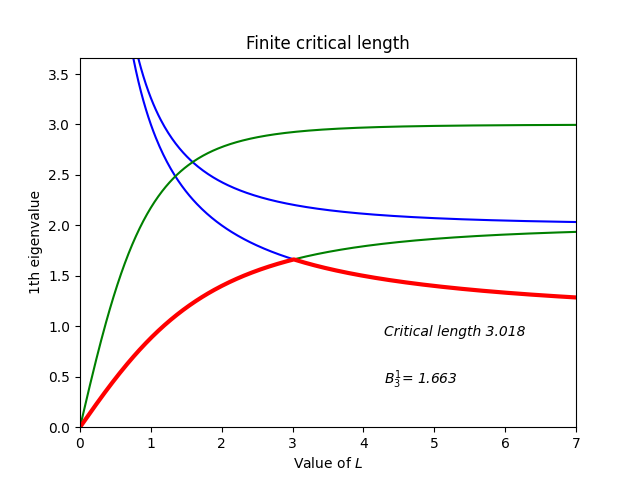}
    \includegraphics[scale=0.455]{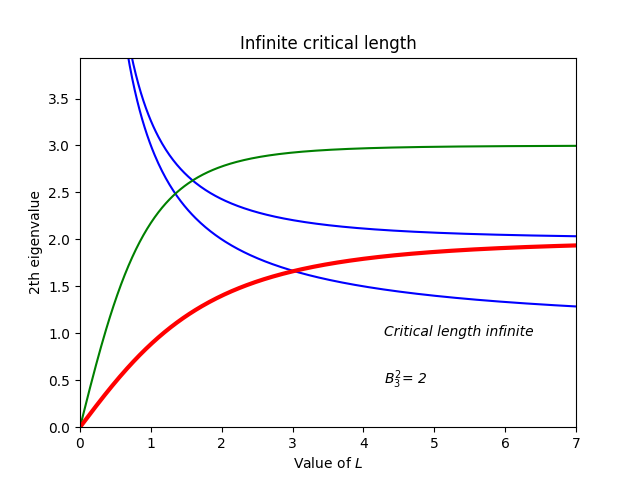}
\end{figure}
\begin{figure}[H]
    \centering
    \includegraphics[scale=0.455]{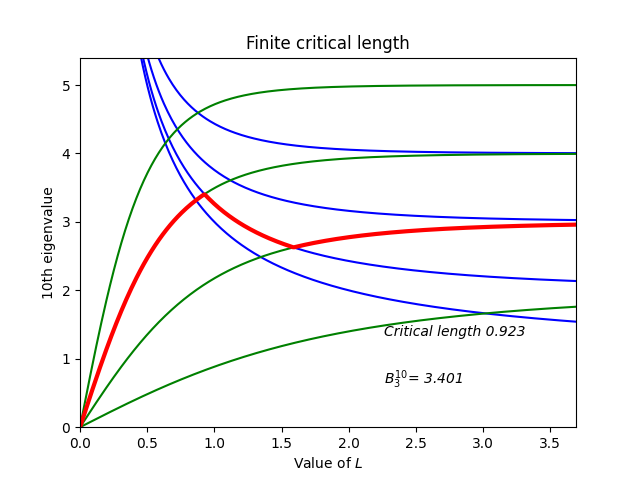}
    \includegraphics[scale=0.45]{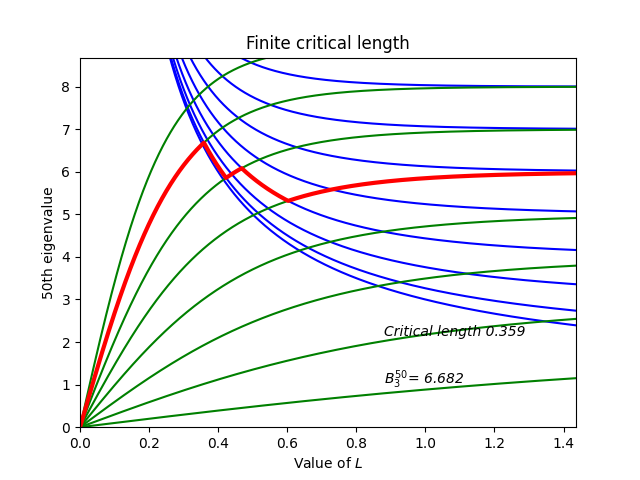}
\end{figure}
\begin{figure}[H]
    \centering
    \includegraphics[scale=0.45]{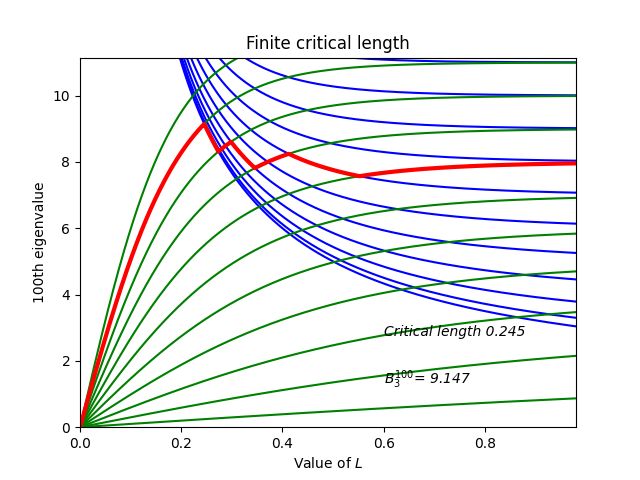}
    \includegraphics[scale=0.45]{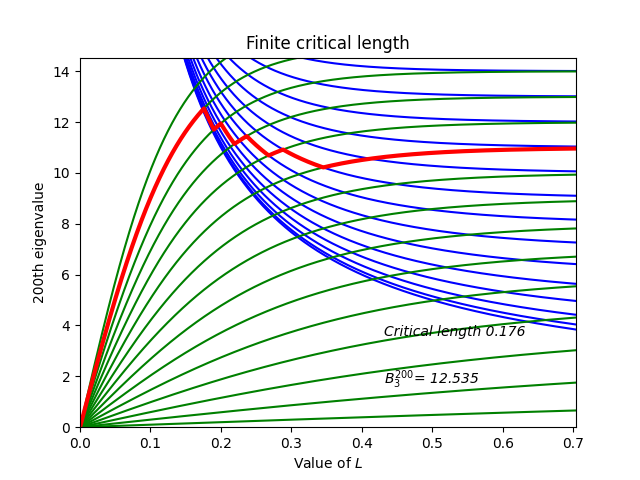}
\end{figure}

\begin{figure}[H]
     \centering
    \includegraphics[scale=0.455]{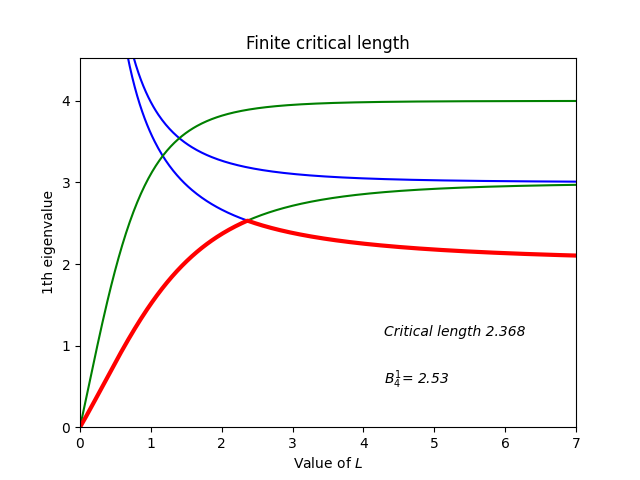}
    \includegraphics[scale=0.455]{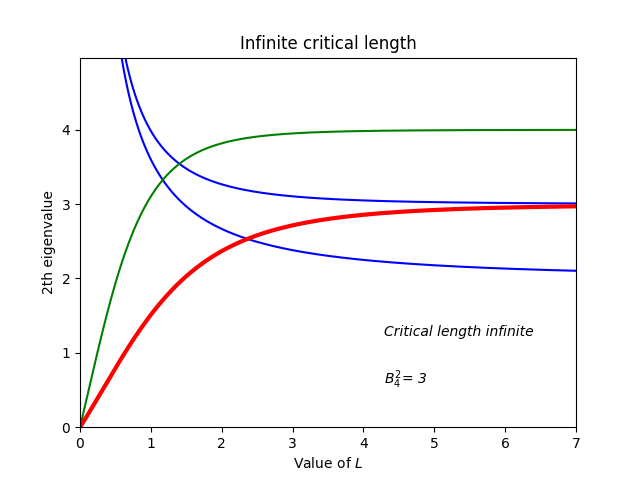}
\end{figure}
\begin{figure}[H]
    \centering
    \includegraphics[scale=0.455]{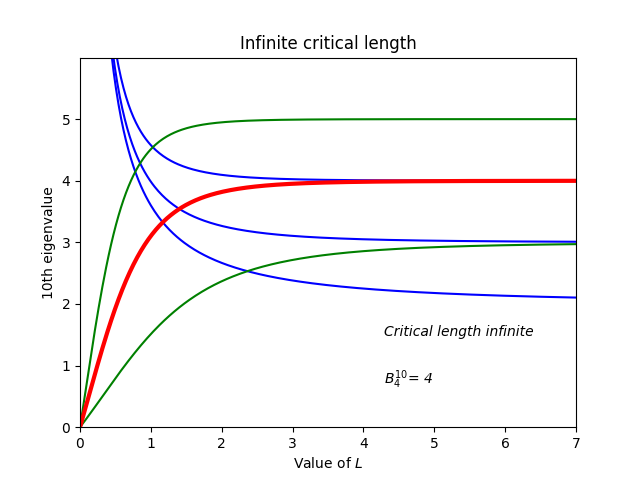}
    \includegraphics[scale=0.45]{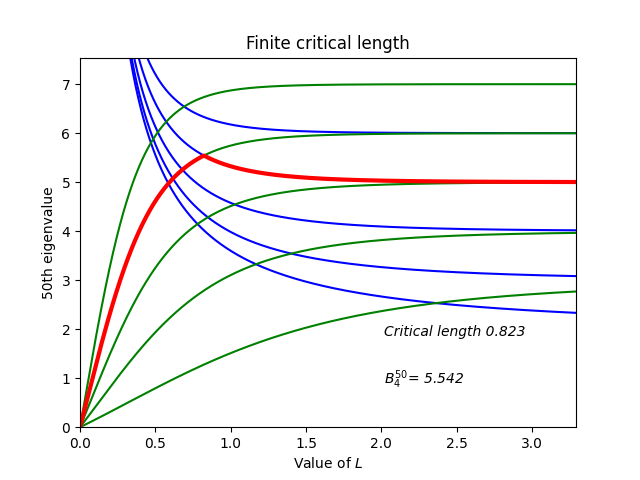}
\end{figure}
\begin{figure}[H]
    \centering
    \includegraphics[scale=0.45]{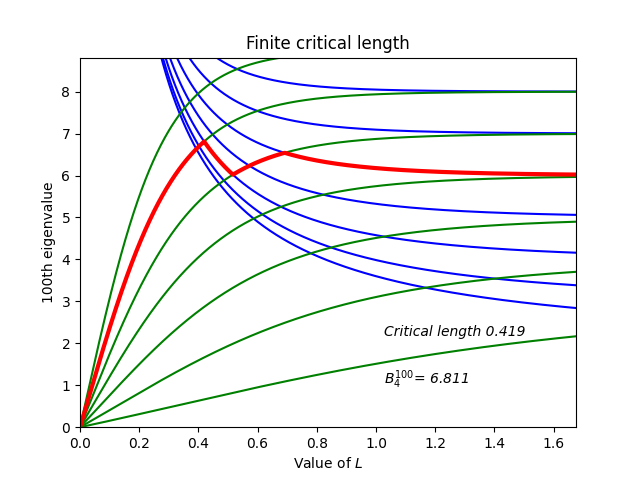}
    \includegraphics[scale=0.45]{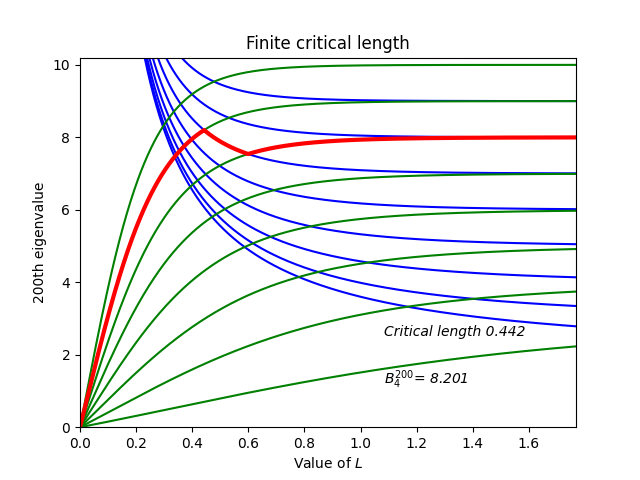}
\end{figure}

\begin{figure}[H]
     \centering
    \includegraphics[scale=0.455]{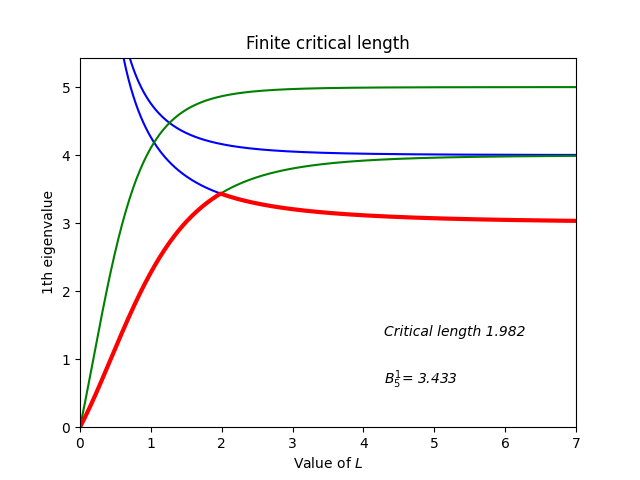}
    \includegraphics[scale=0.455]{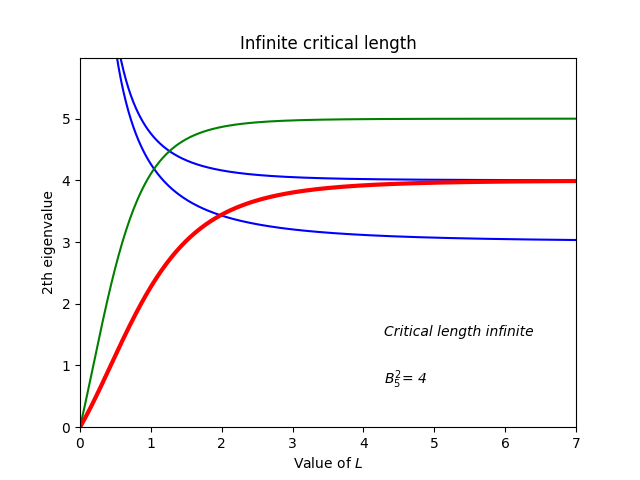}
\end{figure}
\begin{figure}[H]
    \centering
    \includegraphics[scale=0.455]{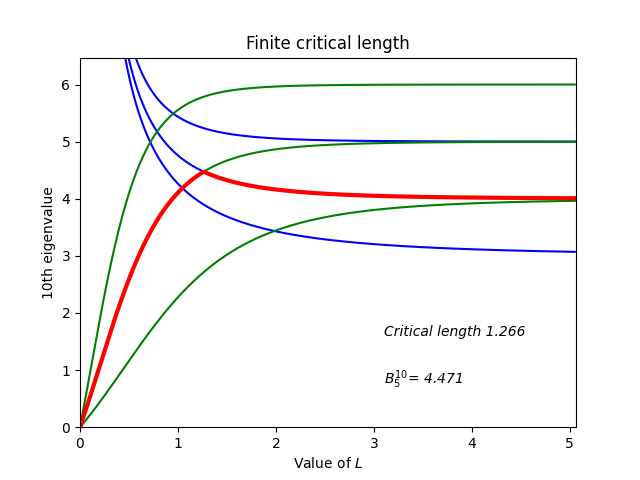}
    \includegraphics[scale=0.45]{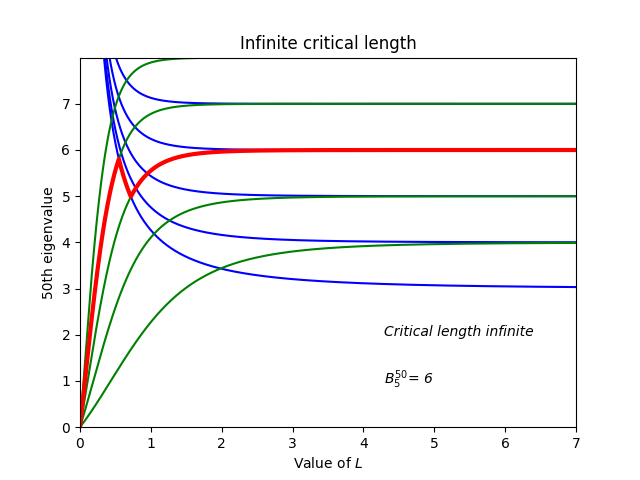}
\end{figure}
\begin{figure}[H]
    \centering
    \includegraphics[scale=0.45]{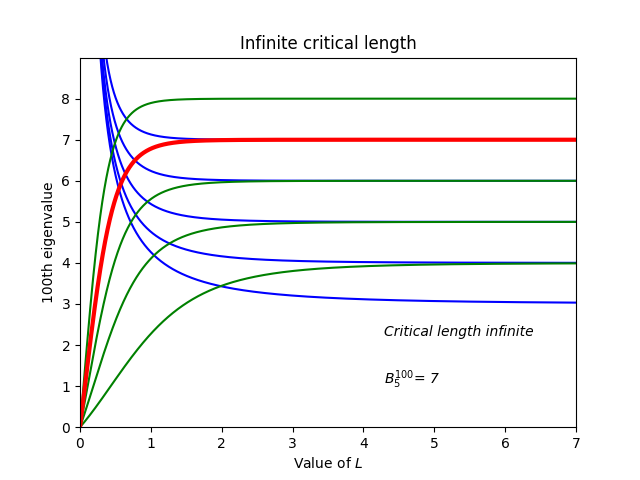}
    \includegraphics[scale=0.45]{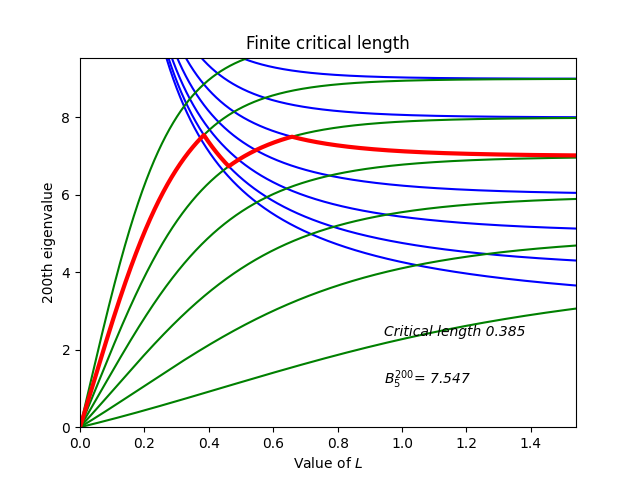}
     \caption[Sharp upper bounds with mixed eigenvalues]{The blue curves are the Steklov-Dirichlet eigenvalues, the green ones are the Steklov-Neumann eigenvalues.}
\end{figure}


\section{Codes} \label{appendix : python codes}
The code that we developed can be found on GitHub, along with other documents such as results, computation times, etc., at the following address: 
\begin{center}
    \href{https://github.com/Tchatchu5}{https://github.com/Tchatchu5}.
\end{center}

For the Python version of the codes, please see
\begin{center}
    \href{https://github.com/Tchatchu5/Python-Extension-process-for-steklov-eigenvalues}{https://github.com/Tchatchu5/Python-Extension-process-for-steklov-eigenvalues}.
\end{center}

For the Julia version of the codes, please see
\begin{center}
    \href{https://github.com/Tchatchu5/Julia-Extension-process-for-Steklov-eigenvalues}{https://github.com/Tchatchu5/Julia-Extension-process-for-Steklov-eigenvalues}.
\end{center}

\end{appendices}


%
%

\end{document}